\documentclass[12pt]{amsart}

\title{When is the \'etale open topology a field topology?}

\author{Philip Dittmann}
\address{Department of Mathematics, University of Manchester, Manchester M13 9PL, United Kingdom}
\email{philip.dittmann@manchester.ac.uk}

\author{Erik Walsberg}
\address{University of California, Irvine}
\email{ewalsber@uci.edu}
\urladdr{math.uci.edu/~ewalsber/}

\author{Jinhe Ye}
\address{Institut de Mathématiques de Jussieu - Paris Rive Gauche
\newline \indent {\em Current address:}\newline
\indent Mathematical Institute, University of Oxford
}
\email{jinhe.ye@imj-prg.fr, jinhe.ye@maths.ox.ac.uk}
\urladdr{sites.google.com/view/vincentye}

\thanks{JY was partially supported by GeoMod AAPG2019 (ANR-DFG), Geometric and Combinatorial Configurations in Model Theory.
  The writing up of this material profited from discussions while PD and JY participated in the programme ``Definability, Decidability and Computability in Number Theory, Part 2'' hosted by the Mathematical Sciences Research Institute in Berkeley, California, and supported by the US National Science Foundation under Grant No.\ DMS-1928930. We would like to thank Will Johnson for spotting an error in the proof of Proposition~\ref{prop:general} in a previous version.}

\usepackage[utf8]{inputenc}
\usepackage{amsmath, amssymb, amsthm, enumitem, comment}

\usepackage{tikz-cd}
\usepackage[Symbol]{upgreek}
\usepackage[mathscr]{euscript}
\usepackage{fullpage} 	
\usepackage{hyperref}
\usepackage[normalem]{ulem}
\usepackage{lineno}
\usepackage[all]{xy}
\usepackage{centernot}
\usepackage{xcolor}

\DeclareFontFamily{U}{BOONDOX-calo}{\skewchar\font=45 }
\DeclareFontShape{U}{BOONDOX-calo}{m}{n}{
  <-> s*[1.05] BOONDOX-r-calo}{}
\DeclareFontShape{U}{BOONDOX-calo}{b}{n}{
  <-> s*[1.05] BOONDOX-b-calo}{}
\DeclareMathAlphabet{\mathcalboondox}{U}{BOONDOX-calo}{m}{n}
\SetMathAlphabet{\mathcalboondox}{bold}{U}{BOONDOX-calo}{b}{n}
\DeclareMathAlphabet{\mathbcalboondox}{U}{BOONDOX-calo}{b}{n}

\DeclareMathOperator*{\forkindep}{\raise0.2ex\hbox{\ooalign{\hidewidth$\vert$\hidewidth\cr\raise-0.9ex\hbox{$\smile$}}}}

\DeclareFontFamily{U}{fsy}{}
\DeclareFontShape{U}{fsy}{m}{n}{<->s*[.9]psyr}{}
\DeclareSymbolFont{der@m}{U}{fsy}{m}{n}
\DeclareMathSymbol{\der}{\mathord}{der@m}{182}

\newcommand{\longpoly}{X^{n+1} + X^n + a_{n-1} X^{n-1} + \ldots + a_1 X + a_0}

\newcommand{\longpolycoeff}{a_0,\ldots,a_{n-1} \in}
\newcommand{\seak}{\Sa E_K}

\newcommand{\Cal}[1]{\ensuremath{\mathcal{#1}}}
\newcommand{\polyn}{\mathrm{Pol}_n}

\newcommand{\mfrak}{\mathfrak{m}}
\newcommand{\pfrak}{\mathfrak{p}}

\newcommand{\Sa}[1]{\ensuremath{\mathscr{#1}}}

\newcommand{\Spec}{\operatorname{Spec}}

\newcommand{\Frac}{\operatorname{Frac}}

\newcommand{\Chara}{\operatorname{Char}}

\newcommand{\extlk}{\operatorname{Ext}_{L/K}}

\newtheorem*{claim-star}{Claim}
\newtheorem{theorem}{Theorem}[section] 
\newtheorem{lemma}[theorem]{Lemma}

\newtheorem{prop-def}[theorem]{Proposition-Definition}
\newtheorem{corollary}[theorem]{Corollary}
\newtheorem{fact}[theorem]{Fact}
\newtheorem{fact-eh}[theorem]{Fact(?)}

\newtheorem{question}[theorem]{Question}
\newtheorem{proposition}[theorem]{Proposition}
\newtheorem{proposition-eh}[theorem]{Proposition(?)}
\newtheorem*{theorem-star}{Theorem}
\newtheorem*{conjecture-star}{Conjecture}
\newtheorem*{question-star}{Question}
\newtheorem*{lemma-star}{Lemma}

\theoremstyle{definition}
\newtheorem{definition}[theorem]{Definition}

\newtheorem{remark}[theorem]{Remark}
\theoremstyle{remark}

\newcommand{\Aa}{\mathbb{A}}

\newcommand{\Gg}{\mathbb{G}}
\newcommand{\Qq}{\mathbb{Q}}
\newcommand{\Rr}{\mathbb{R}}
\newcommand{\Zz}{\mathbb{Z}}
\newcommand{\Nn}{\mathbb{N}}
\newcommand{\Cc}{\mathbb{C}}
\newcommand{\Pp}{\mathbb{P}}

\newcommand{\cE}{\mathscr{E}}

\newenvironment{claimproof}[1][\proofname]
               {
                 \proof[#1]
                 
               }
               {
                 \endproof
               }

\begin{document}

\maketitle

\begin{abstract}
  We investigate the following question: Given a field $K$, when is the \'etale open topology $\cE_K$ induced by a field topology? On the positive side, when $K$ is the fraction field of a local domain $R\neq K$, using a weak form of resolution of singularities due to Gabber, we show that $\cE_K$ agrees with the $R$-adic topology when $R$ is quasi-excellent and henselian.
  Various pathologies appear when dropping the quasi-excellence assumption.
  For locally bounded field topologies, we introduce the notion of generalized t-henselianity (gt-henselianity) following Prestel and Ziegler. We establish the following:
  For a locally bounded field topology $\uptau$, the \'etale open topology is induced by $\uptau$ if and only if $\uptau$ is gt-henselian and some non-empty \'etale image is $\uptau$-bounded open.
  On the negative side, we obtain that for a pseudo-algebraically closed field $K$, $\cE_K$ is never induced by a field topology.
\end{abstract}

\section{Introduction}
\noindent
We continue the study of the \emph{étale open topology}, initiated in \cite{firstpaper} and continued in \cite{secondpaper} and \cite{thirdpaper}.
Recall that the étale topology for a field $K$, also called $\Sa E_K$, is given by a topology on the set of rational points $V(K)$ for every $K$-variety $V$ (a \emph{system of topologies} in the terminology of \cite{firstpaper});
concretely, the $\Sa E_K$-topology on $V(K)$ is defined to have as a basis the collection of sets $f(W(K))$, where $W$ is another $K$-variety and $f \colon W \to V$ is an étale morphism.

The étale open topology is only interesting in the case of fields which are \emph{large} in the sense of Pop (see \cite{Pop-little}) but not separably closed, since otherwise it degenerates to the discrete topology or the Zariski topology, respectively.
Under this restriction, however, the abstract definition coincides with familiar topologies in many cases:
Notably, over the fields $\Cc, \Rr, \Qq_p$ we recover on each variety the Zariski topology, resp.\ real topology, resp.\ $p$-adic topology.
In particular, for $\Rr$ and $\Qq_p$, the étale open topology on every variety is induced by a Hausdorff non-discrete field topology on the ground field.

To generalize the phenomenon on $\Rr$ or $\Qq_p$, consider a local domain $R \subsetneq K$ with fraction field $K$, and recall that the $R$-adic topology on $K$ is the field topology with basis $\{ a R + b \colon a \in K^\times, b \in K \}$.
Like any other field topology, this induces a topology on $V(K)$ for any $K$-variety $V$, which we also call the $R$-adic topology.
If $R$ is a (non-trivial) valuation ring, then the $R$-adic topology is the usual valuation topology.

We now have the following facts relating $\Sa E_K$ and $R$-adic topologies.
\begin{fact}\label{fact:intro-old-R-adic}
  Let $R \subsetneq K$ be a local domain with fraction field $K$.
  \begin{enumerate}
  \item \cite[Theorem 1.2]{thirdpaper} If $R$ is henselian\footnote{We
      recall the definition of a henselian local ring below.
    For a valuation ring, this agrees with the usual notion of henselianity.}
    then the $R$-adic topology refines $\Sa E_K$.
  \item \cite[Theorem 6.15]{firstpaper} If $R$ is a valuation ring and the Henselization of $K$ with respect to the corresponding valuation is not separably closed, then $\Sa E_K$ refines the $R$-adic topology.
  \item \cite[Theorem 1.2]{thirdpaper} If $R$ is a regular (in the sense of commutative algebra) then $\Sa E_K$ refines the $R$-adic topology.
  \end{enumerate}
\end{fact}
\noindent
Here by the $R$-adic topology refining $\Sa E_K$ or vice versa we mean that the corresponding topologies on $V(K)$ refine each other for every variety $V/K$.
(Note, however, that this is equivalent to merely saying that the same holds only on $K^n = \mathbb{A}^n(K)$ for every $n$, see Fact \ref{fact:refine} below.)
The present paper is motivated by the following natural questions:

\begin{question}\label{question:etale-open-field-top}
~
  \begin{enumerate}
  \item When is the $\Sa E_K$-topology induced by a field topology?
  \item When does the $\Sa E_K$-topology agree with the $R$-adic topology for a local domain $R \subsetneq K$ with fraction field $K$?
  \end{enumerate}
\end{question}
\noindent
We prove that the $\Sa E_K$ is not induced by a field topology when $K$ is a pseudo-algebraically closed (PAC) field (Proposition \ref{prop:pac} below), answering a question posed in \cite[Section~8]{firstpaper}.
Since ``most'' algebraic extensions of $\Qq$ in a suitable sense are PAC, see \cite[Proposition~1]{DittmannFehm-nondefinability}, this shows that the ``generic'' answer to Question \ref{question:etale-open-field-top}(1) is negative.

In the other direction, we extend Fact \ref{fact:intro-old-R-adic} to quasi-excellent local domains, a wide class of non-pathological Noetherian domains:
\begin{theorem}[{Theorem \ref{thm:resolution-2}}]
  If $R \subsetneq K$ is a quasi-excellent local domain with fraction field $K$, then the $\Sa E_K$-topology refines the $R$-adic topology.
\end{theorem}
\noindent
Together with Fact \ref{fact:intro-old-R-adic}(1), we deduce:
\begin{corollary}\label{cor:quasiexcellent-henselian}
  If $R \subsetneq K$ is quasi-excellent henselian local domain (e.g.\ $R$ a complete Noetherian local domain) with fraction field $K$, then the $R$-adic topology coincides with the $\Sa E_K$-topology.
\end{corollary}

\noindent
In the case of a $1$-dimensional Noetherian henselian local domain $R$, we can even characterize precisely when the $R$-adic topology coincides with the étale open topology on the fraction field, see Corollary~\ref{cor:japan}.

In Sections \ref{section:refinement} and \ref{section:example} we give examples of pathologies that can arise when the quasi-excellence assumption is dropped, exhibiting at the same time interesting behaviour of the étale open topology under finite field extensions.

Finally, to study Question \ref{question:etale-open-field-top} in much greater generality, we borrow the model-theoretic tools of \cite{Prestel1978}.
This allows to obtain comprehensive answers at least up to replacing the field $K$ by a suitable elementary extension.

In this vein, it had previously been shown \cite[Theorem B]{firstpaper} that the $\Sa E_K$-topology for $K$ not separably closed is induced by a so-called V-topology on $K$ if and only if $K$ is a so-called t-henselian field, i.e.\ if and only if some elementary extension $K^\ast \succ K$ carries a henselian valuation.

We study a notion of \emph{gt-henselian} field topologies, a natural generalization of the notion of a t-henselian field topology from \cite{Prestel1978}.
In fact, this notion agrees with a different notion of henselianity for rings suggested (but hardly studied) in the literature, see Remark \ref{rem:gt-henselian-literature}.
When $\Sa E_K$ is induced by a field topology, that topology must necessarily be gt-henselian (Lemma~\ref{lem:azu-refine}).
We then obtain the following answer to Question \ref{question:etale-open-field-top} with the restriction to locally bounded field topologies:
\begin{theorem}[{Proposition \ref{prop:azumayan}}]
  Suppose that $\uptau$ is a locally bounded field topology on $K$.
Then $\uptau$ induces the $\seak$-topology if and only if $\uptau$ is gt-henselian and some nonempty \'etale image in $K$ is $\uptau$-bounded.
\end{theorem}
\begin{theorem}\label{thm:intro-loc-bdd-henselian-R-adic}
  The $\Sa E_K$-topology is induced by a locally bounded field topology if and only if there exists an elementary extension $K^\ast \succ K$ and a henselian local domain $R \subsetneq K^\ast$ with fraction field $K^\ast$ such that the $R$-adic topology induces $\Sa E_{K^\ast}$.
\end{theorem}
\noindent
It remains open whether the $\Sa E_K$-topology can ever be induced by a field topology which is not locally bounded.

\section{Conventions and background}
\label{section:conventions and background}
\noindent
Throughout, $K$ is a field and $\Chara(K)$ its characteristic.

\subsection{Scheme theory}
A $K$-variety is a separated $K$-scheme of finite type, not necessarily irreducible or reduced.
(This is the convention of \cite[Definition 2.1.1]{poonen-qpoints}.)
Throughout $\Aa^n$ is $n$-dimensional affine space over $K$, i.e. $\Aa^n = \Spec K[X_1,\ldots,X_n]$.
We let $V(K)$ be the set of $K$-points of a $K$-variety $V$.
Given a scheme $W$ we let $\mathcal{O}_W$ be the structure sheaf of $W$, $\mathcal{O}_{W,p}$ be the local ring of $W$ at $p \in W$, and let $\mathcal{O}_p = \mathcal{O}_{W,p}$ when $W$ is clear.

\subsection{The \'etale open topology}
We gather basic facts on the $\seak$-topology from~\cite{firstpaper}.
For a $K$-variety $W$, the \emph{étale open topology} on the set of rational points $W(K)$ is the topology with basis the collection of \emph{étale images} $U \subseteq W(K)$, i.e.\ the sets $U = f(V(K))$ where $f \colon V \to W$ is an étale morphism.
We also write $\Sa E_K$ for the étale open topology on $W(K)$ for any $W$, with $W$ generally clear from context.

\begin{fact}
\label{fact:basic}
Suppose that $V \to W$ is a morphism of $K$-varieties.
Then:
\begin{enumerate}
\item the induced map $V(K) \to W(K)$ is $\Sa E_K$-continuous.
\item if $V \to W$ is \'etale then the induced map $V(K) \to W(K)$ is $\Sa E_K$-open.
\item the map $K \to K$, $x \mapsto \alpha x + \beta$ is an $\Sa E_K$-homeomorphism for any $\alpha \in K^\times, \beta \in K$.
\item if $n$ is prime to $\Chara(K)$ then $\{ \alpha^n : \alpha \in K^\times \}$ is an \'etale open subset of $K$.
\end{enumerate}
\end{fact}

\begin{proof}
(1), (2) is \cite[Lemma 5.2, 5.3]{firstpaper}, respectively.
(3) follows from (1).
Let $\Gg_m = \Spec K[X,X^{-1}]$ be the scheme-theoretic multiplicative group over $K$.
Then (4) follows from (2) as the morphism $\Gg_m \to \Gg_m$, $X \mapsto X^n$ is \'etale when $n$ is prime to $\Chara(K)$.
\end{proof}

\noindent
Suppose that $L$ is an extension of $K$ and $V$ is a $K$-variety.
We let $V_L = V \times_{\Spec K} \Spec L$ be the base change of $V$.
Recall that $V_L(L)$ is canonically identified with $V(L)$, so we canonically equip $V(L)$ with the $\Sa E_L$-topology.
Fact~\ref{fact:restrict} below is \cite[Theorem~5.8]{firstpaper}.

\begin{fact}
\label{fact:restrict}
Suppose that $L$ is an algebraic extension of $K$ and $V$ is a $K$-variety.
Then the $\Sa E_K$-topology on $V(K)$ refines the topology induced on $V(K)$ by the $\Sa E_L$-topology on $V(L) = V_L(L)$, i.e. if $O \subseteq V(L)$ is $\Sa E_L$-open then $O \cap V(K)$ is $\Sa E_K$-open.
\end{fact}

\subsection{Ring topologies and field topologies}

Our general reference for ring topologies and field topologies is \cite{Prestel1978}, and we follow its conventions.
In particular, \textbf{ring topologies are always taken to be Hausdorff and not discrete}.

We have the following basic fact about comparisons between the étale open topology and a field topology, proven in \cite[Lemma~4.8, Lemma~4.2]{firstpaper}.

\begin{fact}
\label{fact:refine}
Suppose that $\uptau$ is a field topology on $K$.
If the $\uptau$-topology on each $K^n = \Aa^n(K)$ refines the $\Sa E_K$-topology, then the $\uptau$-topology on $V(K)$ refines the $\Sa E_K$-topology for any $K$-variety $V$.
If the $\Sa E_K$-topology on $K$ refines $\uptau$, then the $\Sa E_K$-topology on $V(K)$ refines the $\uptau$-topology for any $K$-variety $V$.
\end{fact}
\noindent
Let $R$ be a domain with fraction field $K = \Frac(R)$, and assume $R \neq K$.
The \emph{$R$-adic topology} on $K$ is the topology with basis $\{ a R + b \colon a \in K^\times, b \in K \}$.
This is a ring topology.
(Compare \cite[Example 1.2]{Prestel1978}, although the name $R$-adic topology is not used there.)
We are chiefly but not exclusively interested in the situation where $R$ is local.

We let $J(R)$ be the Jacobson radical of $R$.
It is the intersection of all maximal ideals of $R$, or equivalently $J(R) = \{ x \in R \colon 1 + xR \subseteq R^\times \}$.

\begin{fact}
\label{fact:jacobson-field}
The $R$-adic topology on $K$ is a field topology if and only if $J(R)\ne\{0\}$.
\end{fact}

\begin{proof}
The right to left implication is~\cite[Proposition 3.1]{dp-finite-v}.
We prove the left to right implication.
Suppose that the $R$-adic topology is a field topology.
Then inversion gives a continuous map $K^\times \to K^\times$.
Hence there is nonzero $\alpha \in R$ such that $(1 + \alpha R)^{-1} \subseteq R$. Thus $(1+\alpha R)\subseteq R^\times$ and $\alpha \in J(R)$.
\end{proof}

Fact~\ref{fact:refine-1} follows from Fact~\ref{fact:refine}, \ref{fact:jacobson-field} and the definitions.
\noindent
We leave the details to the reader.
\begin{fact}
\label{fact:refine-1}
Suppose that $R$ has nonzero Jacobson radical (so the $R$-adic topology is a field topology.)
The following are equivalent:\begin{enumerate}
\item The $\Sa E_K$-topology on $V(K)$ refines the $R$-adic topology for any $K$-variety $V$.
\item $R$ contains a nonempty $\Sa E_K$-open subset of $K$.
\end{enumerate}
\end{fact}

\noindent
Given a ring topology $\uptau$ on $K$, a set $B \subseteq K$ is called \emph{bounded} if for every neighbourhood $U$ of $0$ there exists a neighbourhood $V$ of $0$ such that $V \cdot B \subseteq U$.
The topology $\uptau$ is \emph{locally bounded} if there exists a bounded neighbourhood of $0$.
\begin{fact}
\label{fact:open subring}
Let $\uptau$ be a ring topology on $K$ and $S$ an open subring of $K$.
Then $K = \Frac(S)$.
\end{fact}

\begin{proof}
Suppose that $\alpha \in K$ and $\alpha \notin \Frac(S)$.
Note that $S$ is a neighbourhood of zero.
Then $S \cap \alpha S = \{0\}$, hence $\uptau$ is discrete, contradiction.
\end{proof}

\begin{fact}
\label{fact:open subring 2}
Let $\uptau$ be a ring topology on $K$ and $S$ a bounded open subring of $K$.
Then $\uptau$ is the $S$-adic topology.
\end{fact}
\noindent
Therefore the $R$-adic topologies are exactly the ring topologies which admit bounded open subrings.

\begin{proof}

Since $\{\alpha S+\beta: \alpha \in K^\times, \beta \in K\}$ is a basis for the $S$-adic topology and $S$ is open in $\uptau$, $\uptau$ is finer than the $S$-adic topology.
By boundedness of $S$ and non-discreteness of the $\uptau$-topology, for every $\uptau$-open neighbourhood $U$ of $0$ there exists an $\alpha \in K^\times$ with $\alpha S \subseteq U$.
This implies that the $S$-adic topology refines $\uptau$.
\end{proof}

\subsection{Commutative algebra}
\label{section:all those definitions}

Let $R$ be local with maximal ideal $\mfrak$.
Then $R$ is \textbf{henselian} if for any $f \in R[X]$ and $\alpha \in R$ with $f(\alpha) \equiv 0 \not\equiv f'(\alpha) \pmod{\mfrak}$ there is $\alpha^* \in R$ such that $f(\alpha^*) = 0$ and $\alpha^* \equiv \alpha \pmod{\mfrak}$.

\begin{fact}
\label{fact:hensel-equiv}
The following are equivalent for a local domain $R$ with maximal ideal $\mfrak$.
\begin{enumerate}
\item $R$ is henselian,
\item If $a_0,\ldots,a_{n - 1} \in \mfrak$ then $X^{n+1} - X^n + a_{n - 1} X^{n - 1} + \ldots + a_1 X + a_0$ has a root in $\mfrak + 1$.
\item If $a_0,\ldots,a_{n - 1} \in \mfrak$ then $X^{n+1} + X^n + a_{n - 1} X^{n - 1} + \ldots + a_1 X + a_0$ has a root in $\mfrak - 1$.
\end{enumerate}
\end{fact}

\begin{proof}
(1)$\Leftrightarrow$(2) is in \cite[Proposition 1]{Gabber-K-theory}.
(2)$\Leftrightarrow$(3) follows by considering the substitution $Y = -X$.
\end{proof}
\noindent
We gather some more intricate notions from commutative algebra, for use in Sections \ref{section:N-1}, \ref{section:excellent} and \ref{section:refinement}.
Let $S$ be a ring.
We let $\dim S$ be the Krull dimension of $S$.
If $S$ is local then $S$ is \textbf{regular} if $S$ is Noetherian and the maximal ideal of $S$ is generated by $\dim S$ elements.
This is a notion of non-singularity.
A locally Noetherian scheme is defined to be regular if all its stalks are regular local rings, and a Noetherian ring $R$ is defined to be regular if $\operatorname{Spec} R$ is regular, i.e.\ if all localizations of $R$ at prime ideals are regular local rings.

\begin{fact}
\label{fact:japan}
Suppose that $R$ is a one-dimensional Noetherian  domain, $K$ is the fraction field of $R$, and $S$ is the integral closure of $R$ in $K$.
Then $S$ is a regular ring.
\end{fact}

\begin{proof}
By Krull-Akizuki~\cite[Tag~00PG]{stacks-project} $S$ is Noetherian, and by \cite[Tag~00OK]{stacks-project} $S$ is one-dimensional.
A one-dimensional Noetherian normal domain is a Dedekind domain, hence regular~\cite[Tag~034X]{stacks-project}.
\end{proof}

\noindent
Let $R$ be a domain with fraction field $K$ and $S$ the integral closure of $R$ in $K$.
Then $R$ is \textbf{normal} if $R = S$, $R$ is $N$-1 if $S$ is a finite $R$-module, and $R$ is \textbf{Japanese} (or $N$-2) if the integral closure of $R$ in any finite field extension of $K$ is a finite $R$-module.
Non-Japanese Noetherian rings are viewed as~pathologies.

We now discuss quasi-excellent rings, a class of Noetherian rings, and the related slightly more restrictive class of excellent rings.
The definitions in full generality are somewhat technical, so we omit them.
We direct the readers to~\cite[Tag~07QT, 07GH, 07P7, 00NL]{stacks-project} for the definitions and to \cite{Rotthaus-excellent} for a friendlier introduction, as well as \cite[Exposé I]{TravauxDeGabber} for a comprehensive overview.
The class of excellent rings excludes certain pathologies that can arise for general Noetherian rings, but nevertheless includes virtually all ``naturally occurring'' Noetherian rings.

We give a definition of quasi-excellence for local rings.
Suppose that $L$ is a field and $R$ is an $L$-algebra.
Then $R$ is \textbf{geometrically regular} if $R \otimes_L L^\mathrm{alg}$ is regular, where $L^\mathrm{alg}$ is an algebraic closure of $L$.
Regularity implies geometric regularity when $L$ is perfect.
A morphism $R \to S$ of Noetherian rings is regular if $R \to S$ is flat and $S \otimes_{R} \Frac(R/\pfrak)$ is geometrically regular over $\Frac(R/\pfrak)$ for every prime ideal $\pfrak$ in $R$.
In scheme-theoretic language $R \to S$ is regular if it is flat and every scheme-theoretic fiber of $\Spec S \to \Spec R$ is geometrically regular.

\begin{fact}
  \label{fact:good-char}
  Let $S$ be a Noetherian local ring.
  \begin{enumerate}[leftmargin=*]
  \item $S$ is quasi-excellent if and only if $S \to \widehat{S}$ is regular, where $\widehat{S}$ is the completion.
  \item If $S$ is either normal or henselian, then $S$ is quasi-excellent if and only if it is excellent.
  \end{enumerate}
\end{fact}

\noindent
See the discussion in \cite[Section 34]{matsumura} or \cite[Exposé I, Proposition 5.5.1 (ii)]{TravauxDeGabber} for Fact~\ref{fact:good-char}(1) and \cite[Corollary 2.3]{heinzer-rotthaus-wiegand-catenary-local} or \cite[Tag~0C2F]{stacks-project} for Fact~\ref{fact:good-char}(2).
We may take regularity of $S \to \widehat{S}$ to be the definition of quasi-excellence for Noetherian  local rings.
Note that complete local Noetherian rings are trivially excellent by this definition.
We collect some general facts.
\begin{fact}
\label{fact:ring-properties}
\hspace{.1cm}
\begin{enumerate}
[leftmargin=*]
\item The class of normal rings is closed under localizations.
\item The class of quasi-excellent rings is closed under finite extensions, localizations, and quotients.
\item Complete local rings are excellent.
\item Quasi-excellent rings are Japanese.
\item The class of henselian local rings is closed under quotients.
\item The Henselization of a quasi-excellent local ring is quasi-excellent.
\item If $R$ is $N$-1 and $\Chara(K) = 0$ then $R$ is Japanese.
\end{enumerate}
\end{fact}

\begin{proof}
(1) is \cite[Tag~00GY]{stacks-project} and (2) is \cite[Tag~07QU]{stacks-project}.
(3) follows from Fact~\ref{fact:good-char}.
(4) is \cite[Tag~07QV]{stacks-project}.
(5) follows easily from the definitions.
(6) is \cite[Corollaire 18.7.6]{EGA-IV-4}.
(7) is \cite[Tag 032M]{stacks-project}.
\end{proof}

\begin{remark}\label{rmk:excellent-eg}
We now give some examples of excellent (in particular quasi-excellent) henselian local rings, most of which arise as local rings in various kinds of tame spaces.
Let $L$ be a field.
\begin{enumerate}[leftmargin=*]

\item Henselizations of localizations of finitely generated $L$-algebras are excellent.
  In particular the local ring \[L[[t_1,\ldots,t_n]]_\mathrm{alg} = \{ p \in L[[t_1, \dotsc, t_n]] \colon p \text{ algebraic over } L(t_1, \dotsc, t_n) \}\] is excellent.
(This is the Henselization of the localization of $L[t_1, \dotsc, t_n]$ at the maximal ideal $(t_1, \dotsc, t_k)$:
henselianity of $L[[t_1, \dotsc, t_n]]_{\mathrm{alg}}$ follows immediately from henselianity of $L[[t_1, \dotsc, t_n]]$,
and conversely the Henselization of $L[t_1, \dotsc, t_n]$ at $(t_1, \dotsc, t_n)$ is algebraically closed in the completion \cite[Corollary 44.3]{nagata-local}.)
When $L$ is real closed $L[[t_1,\ldots,t_n]]_\mathrm{alg}$ is the ring of germs of $n$-variable Nash functions at the origin~\cite[Corollary 8.1.6]{real-algebraic-geometry}.
  \item Complete Noetherian local rings, such as $L[[t_1, \dotsc, t_k]]$ and its quotients, are excellent.
  \item If $L$ is complete with respect to a norm the ring of covergent power series $L\{t_1,\ldots,t_n\}$ in $n$-variables is an excellent local ring. (See \cite[Theorem 45.5]{nagata-local} for henselianity,  \cite[(34.B)]{matsumura} for excellence in the case of $L = \mathbb{R}$ or $L = \mathbb{C}$, and \cite[Théorème 2.13]{ducros-excellent} for excellence in the non-archimedean case.)
Quotients of $\Cc\{t_1,\ldots,t_n\}$ arise as local rings of complex analytic varieties and when $L$ is non-archimedean quotients of $L\{t_1,\ldots,t_n\}$ arise as local rings of Berkovich spaces, see~\cite{ducros-excellent}.
\end{enumerate}
\end{remark}

\begin{fact}
  \label{fact:normalization}
  Let $R$ be a Noetherian henselian local domain.
  Then the integral closure $S$ of $R$ in a finite extension $L$ of its fraction field $K$ is also a henselian local domain.
\end{fact}

\begin{proof}
  The integral closure $S$ is a direct limit of domains which are finite over $R$.
  Any domain finite over $R$ is itself a henselian and local by the characterization \cite[Tag~04GG (10)]{stacks-project} of henselianity, and the class of henselian local domains is closed under direct limits.
\end{proof}

\subsection{Resolution of Singularities}
\label{section:resolution}
A \textbf{resolution of singularities} of a reduced Noetherian scheme $W$ is given by a regular scheme $V$ and a proper birational morphism $V \to W$.
A resolution of singularities of a Noetherian ring $R$ is a resolution of singularities of $\Spec R$.

Fact~\ref{fact:one-dim-resolve} is related to the fact that a one-dimensional reduced $K$-variety admits a resolution of singularities.

\begin{fact}
\label{fact:one-dim-resolve}
Suppose that $R$ is a one-dimensional Noetherian domain and let $S$ be the integral closure of $R$ in $K = \operatorname{Frac}(R)$.
Then the following are equivalent:
\begin{enumerate}
\item $\Spec R$ admits a resolution of singularities.
\item $R$ is $N$-1 (i.e. $S$ is a finite $R$-module).
\item the natural morphism $\uppi\colon\Spec S \to \Spec R$ is a resolution of singularities for $\Spec R$.
\end{enumerate}
\end{fact}

Recall that if $T$ is a finite extension of $R$ in $K$ then $\Spec T \to \Spec R$ is birational.

\begin{proof}
By Fact~\ref{fact:japan} $S$ is regular.
If $(2)$ holds then $\uppi$ is finite, hence proper and birational.
So (2) implies (3).
Clearly (3) implies (1).
See \cite[Section 2.4, p.~11, last paragraph before Exercise 2.15]{cutkosky} for a proof that (1) implies (2).
\end{proof}

\noindent
Fact~\ref{fact:hironaka} is a famous theorem of Hironaka \cite{Hironaka} (cited as in \cite[1.2 (i)]{Temkin_InsepLocalUnif}).

\begin{fact}
\label{fact:hironaka}
Suppose that $R$ is a quasi-excellent local domain of residue characteristic zero.
Then any reduced scheme of finite type over $R$ admits a resolution of singularities.
In particular $R$ admits a resolution of singularities.
\end{fact}
\noindent
Fact~\ref{fact:hironaka} in positive residue characteristic is of course an open conjecture~\cite[7.9.6]{EGAIV-2}.
We use a weaker form of resolution of singularities due to Gabber.
Suppose that $R$ is a Noetherian domain.
An \textbf{altered local uniformization}\footnote{The
  terminology is borrowed from \cite[1.2 (iv)]{Temkin_InsepLocalUnif}.
}
of $R$ consists of regular integral schemes $V_1, \dotsc, V_n$ and generically finite dominant morphisms $V_i \to \Spec R$ of finite type such that every valuation ring $\mathcal{O}$ containing $R$ can be prolonged to a valuation ring $\mathcal{O}^*$ centered on some $V_i$, i.e.\ there exists a commutative diagram as follows:
\[ \xymatrix{
    \Spec \mathcal{O}^* \ar[r] \ar[d] & V_i \ar[d] \\
    \Spec \mathcal{O} \ar[r] & \Spec R
}\]
The valuative criterion for properness implies that a resolution of singularities is an altered local uniformization.

\begin{theorem}[Gabber]
\label{thm:quasiExcellentResolutionByAlterations}
A quasi-excellent domain admits an altered local uniformization.
\end{theorem}
\begin{proof}
  By \cite[Exposé~VII, Théorème~1.1]{TravauxDeGabber}, there are regular integral schemes $V_1,\ldots,V_n$ and finite type morphisms $\uppi_i\colon V_i \to \Spec R$ such that $\uppi_1,\ldots,\uppi_n$ are a covering family in the Grothendieck topology of alterations \cite[Exposé~II, 2.3.3]{TravauxDeGabber}.
In particular, each $\uppi_i$ is dominant and generically finite.
The prolongation property for valuation rings follows from \cite[Exposé IV, Théorème 4.2.1]{TravauxDeGabber}.
\end{proof}
\noindent
There are non-quasi-excellent Noetherian local domains which admit an altered local uniformization.
For instance, this is trivially the case for regular local rings which are not quasi-excellent, see for example~\cite[Chapter 13 (34.B)]{matsumura}.

\section{The one dimensional case}
\label{section:N-1}
\noindent
Let $R \subsetneq K$ be a domain with fraction field $K$.

\begin{lemma}
\label{lem:adic-refinement}
Suppose that $R^*$ is a domain with $\Frac(R^*) = K$.
The following are~equivalent:
\begin{enumerate}
\item The $R$-adic topology on $K$ refines the $R^*$-adic topology,
\item $R^*$ is $R$-adically open,
\item $\alpha R \subseteq R^*$ for some $\alpha \in K^\times$,
\item $R$ is bounded in the $R^*$-adic topology.
\end{enumerate}
Hence the $R$-adic topology agrees with the $R^*$-adic topology if and only if there are $\alpha,\beta \in K^\times$ such that $\alpha R \subseteq R^*$ and $\beta R^* \subseteq R$, i.e. if $R,R^*$ is $R^*$-,$R$-adically bounded, respectively.
\end{lemma}
\noindent
Lemma~\ref{lem:adic-refinement} follows easily from the definitions, so we leave it to the reader.

\begin{lemma}
\label{lem:finite-module}
Suppose that $R$ is Noetherian and $S$ is a subring of $K$ containing $R$.
Then the following are equivalent:
\begin{enumerate}
\item $S$ is a finite $R$-module,
\item the $R$-adic and $S$-adic topologies on $K$ agree.
\end{enumerate}
\end{lemma}

\begin{proof}
Suppose (2).
By Lemma~\ref{lem:adic-refinement} there is $\alpha \in K^\times$ with $\alpha S \subseteq R$.
As $R$ is Noetherian $\alpha S$ is a finite $R$-module, so $S$ is a finite $R$-module.
Suppose (1).
By Lemma~\ref{lem:adic-refinement} it is enough to show $\alpha S \subseteq R$ for some $\alpha \in K^\times$.
We have $S = \beta_1 R + \ldots + \beta_n R$ for $\beta_1,\ldots,\beta_n \in K$.
Fix $\alpha \in K$ with $\alpha \beta_i \in R$ for all $i$.
Then $\alpha S = (\alpha \beta_1) R + \ldots + (\alpha \beta_n) R \subseteq R$.
\end{proof}
\noindent
Lemma~\ref{lem:N-1} is immediate from Lemma~\ref{lem:finite-module}.

\begin{lemma}
\label{lem:N-1}
Suppose that $R$ is Noetherian and $S$ is the integral closure of $R$ in $K$.
Then the following are equivalent:
\begin{enumerate}
\item $R$ is $N$-1,
\item the $R$-adic and $S$-adic topologies on $K$ agree.
\end{enumerate}
\end{lemma}
\noindent
Proposition~\ref{prop:N-1} is a partial converse to our theorem that if $R$ is an excellent henselian local domain then the $R$-adic and $\seak$-topologies agree.
(Recall that an excellent ring is $N$-1.)

\begin{proposition}
\label{prop:N-1}
Suppose that $R$ is a Noetherian, henselian, and local,  and the $\Sa E_K$-topology agrees with the $R$-adic topology.
Then $R$ is $N$-$1$.
\end{proposition}

\begin{proof}
Let $S$ be the integral closure of $R$ in $K$.
As $R \subseteq S$, the $R$-adic topology refines the $S$-adic topology.
By Fact~\ref{fact:normalization} $S$ is a henselian local ring.
By Fact~\ref{fact:intro-old-R-adic}(1) the $S$-adic topology on $K$ refines the $\Sa E_K$-topology.
Hence the $S$-adic topology agrees with the $R$-adic topology.
Apply Lemma~\ref{lem:N-1}.
\end{proof}

\begin{corollary}
\label{cor:japan}
Suppose that $R$ is one-dimensional, Noetherian, and henselian local.
Then the following are equivalent:
\begin{enumerate}
\item the $\Sa E_K$-topology agrees with the $R$-adic topology.
\item $R$ is $N$-1.
\item $\operatorname{Spec} R$ admits a resolution of singularities.
\end{enumerate}
\end{corollary}

\begin{proof}
The equivalence of (2) and (3) is Fact \ref{fact:one-dim-resolve}.
Proposition~\ref{prop:N-1} shows that (1) implies (2).
Suppose that $R$ is $N$-1 and let $S$ be the integral closure of $R$ in $K$.
Then $S$ is itself local by Fact \ref{fact:normalization}.
By Lemma~\ref{lem:N-1} it is enough to show that the $\Sa E_K$-topology agrees with the $S$-adic topology.
By Facts~\ref{fact:normalization} and \ref{fact:intro-old-R-adic}(1) the $S$-adic topology refines the $\Sa E_K$-topology.
By Facts~\ref{fact:japan} and \ref{fact:intro-old-R-adic}(2) the $\Sa E_K$-topology refines the $S$-adic topology.
\end{proof}

\section{The quasi-excellent case}
\label{section:excellent}
\noindent
Let again $R \subsetneq K$ be a domain with fraction field $K$.
The central result of this section is the following theorem, which generalizes Fact \ref{fact:intro-old-R-adic}(3).
\begin{theorem}
\label{thm:resolutionByAlterations}
Suppose that  $R$ is local, normal, and Noetherian.
If $R$ admits an altered local uniformization, then the étale open topology over $K$ refines the $R$-adic topology.
\end{theorem}
\noindent
We gather some lemmas.
Fact~\ref{fact:JL} is a slight generalization given in~\cite[Lemma~4.3]{thirdpaper} of a result of Jensen and Lenz-
ing~\cite[pg 52,55]{model-theoretic-algebra}.

\begin{fact}
\label{fact:JL}
Suppose that $R$ is a regular local domain with maximal ideal $\mfrak$ and $\dim R \ge 2$.
\begin{enumerate}
\item If $\Chara(R/\mfrak) \ne 2$ and $\alpha,\beta \in K$ satisfy $1 + \alpha^4 = \beta^2$ then $\alpha \in R$ or $1/\alpha \in R$.
\item If $\Chara(R/\mfrak) = 2$ and $\alpha,\beta \in K$ satisfy $1 + \alpha^3 = \beta^3$ then $\alpha \in R$ or $1/\alpha \in R$.
\end{enumerate}
\end{fact}

\begin{lemma}
\label{lem:discrete-val}
Suppose that $R$ is a local domain, $O_1,\ldots,O_k$ are discrete valuation
subrings of $K$ with maximal ideals $\mfrak_1,\ldots,\mfrak_k$, respectively.
Let $\mathfrak a = \mfrak_1 \cap \ldots \cap \mfrak_k$.
Suppose that one of the following holds:
\begin{enumerate}
\item $\Chara(K) \ne 2$ and if $\alpha \in \mathfrak a, \beta \in K$ satisfy $1 + \alpha^4 = \beta^2$ then $\alpha \in R$,
\item $\Chara(K) \ne 3$ and if $\alpha \in \mathfrak a, \beta \in K$ satisfy $1 + \alpha^3 = \beta^3$ then $\alpha \in R$,
\end{enumerate}
Then the \'etale open topology over $K$ refines the $R$-adic topology.
\end{lemma}

\begin{proof}
Let us assume that $\Chara(K) \neq 2$.
By Fact~\ref{fact:refine-1} it is enough to show that $R$ has interior in the $\seak$-topology.
Each $\mfrak_i$ is $\Sa E_K$-open by Fact~\ref{fact:intro-old-R-adic}(2).
(The condition there that the henselization of $K$ with respect to $O_i$ is not separably closed holds since $O_i$ is discrete; cf.\ also \cite[Corollary 6.17]{firstpaper}.)
Hence $\mathfrak a$ is $\Sa E_K$-open.
Let $B = \{ \beta^2 : \beta \in K^\times\}$ and $f \colon K \to K$ be given by $f(\alpha) = 1 + \alpha^4$.
By Fact~\ref{fact:basic} $f^{-1}(B)$ is $\Sa E_K$-open.
Note that $f^{-1}(B) \cap \mathfrak a$ is contained in $R$ by $(1)$ in the assumption.
Finally $f^{-1}(B) \cap \mathfrak a$ is nonempty as $f^{-1}(B)$ and each $\mfrak_i$ is an $\Sa E_K$-neighbourhood of zero.
The argument for $\Chara(K) \neq 3$ is analogous.
\end{proof}

\begin{lemma}
\label{lem:ringIntersectionOfNiceValnRings}
Suppose that $R$ is normal, local, and Noetherian, and $\alpha \in K \setminus R$.
Then there exists a valuation ring $\mathcal{O}$ of $K$ dominating $R$ with $\alpha \not\in \mathcal{O}$.
\end{lemma}

\begin{proof}
By normality there is a height one prime ideal $\pfrak$ in $R$ such that $\alpha \notin \Cal O_\pfrak$, see \cite[Chapter 7 (17.H) Theorem 38]{matsumura}.
Then $\Cal O_\pfrak$ is normal by Fact~\ref{fact:ring-properties}(1), and so $\Cal O_\pfrak$ is a DVR since it is a one-dimensional normal local domain \cite[Tag~00PD]{stacks-project}.
By Chevalley's extension theorem there is a valuation ring $\Cal O^*$ of $\Frac({\mathcal{O}/\mathfrak{p}})$ dominating the local ring $R/\mathfrak p$.
Let $\Cal O$ be the valuation ring corresponding to the composition of the places associated to $\Cal O_\pfrak$ and $\Cal O^*$.
Then $\alpha \not\in \Cal O$ since $\Cal O \subseteq \Cal O_\pfrak$, and by construction $\Cal O$ dominates $R$.
\end{proof}

\noindent
We now prove Theorem \ref{thm:resolutionByAlterations}.

\begin{proof} 
Let $\Pi = ( X_i \xrightarrow[]{\uppi_i} \Spec R : i \in \{1,\ldots,n\} )$ be an altered local uniformization of $R$.
We make some definitions and constructions for arbitrary $i \in \{1,\ldots,n\}$.
Recall that $X_i$ is integral and let $K_i$ be the function field of  $X_i$.
Let $\mfrak$ be the maximal ideal of $R$.
Then $\mfrak$ is the closed point of $\Spec R$, hence $\uppi^{-1}_i(\mfrak)$ is a proper closed subset of $X_i$ by dominance of $\uppi_i$.
The set $\uppi^{-1}_i(\mfrak)$ has only finitely many irreducible components, each of which is contained in an irreducible codimension one subset of $X_i$.
Let $A_i$ be a finite set of codimension one points in $X_i$ such that every point in $\uppi^{-1}_i(\mfrak)$ is in the closure of some $p \in A_i$.
By regularity $\Cal O_{X_i,p}$ is a DVR for every $p \in A_i$.
Let $\Cal O_{i,p} = \Cal O_{X_i,p} \cap K$ for every $i \in \{1,\ldots,n\}$ and $p \in A_i$.
Since the extension $K_i/K$ is finite, $\Cal O_{i,p}$ is a (non-trivial) DVR.
Let $\mfrak_{i,p}$ be the maximal ideal of each $\Cal O_{i,p}$.

Now suppose first that  $\Chara(R/\mfrak) \ne 2$, hence $\Chara(K) \ne 2$.
Suppose that $\alpha, \beta \in K$ satisfy $1 + \alpha^4 = \beta^2$ and $\alpha$ is in  $\bigcap_{1 \le i \le n} \bigcap_{p \in A_i} \mfrak_{i,p}$.
By Lemma \ref{lem:discrete-val} it is enough to show that $\alpha \in R$.
We suppose towards a contradiction that $\alpha \not\in R$.
By Lemma \ref{lem:ringIntersectionOfNiceValnRings}, there exists a valuation subring $\mathcal{O}$ of $K$ dominating $R$ with $\alpha \not\in \mathcal{O}$.
By the defining property of altered local uniformizations, there is $i \in \{1,\ldots,n\}$, $p \in X_i$, and a valuation subring $\Cal O^*$ of $K_i$ such that $\Cal O^*$ prolongs $\Cal O$ and $\Cal O^*$ dominates $\Cal O_{X_i,p}$.
Thus $\alpha \notin \Cal O_{X_i,p}$.
Fact~\ref{fact:JL} shows that $1/\alpha \in \Cal O_{X_i,p}$ when $\dim \Cal O_{X_i,p} \ge 2$;
in fact the same holds if $\dim \Cal O_{X_i,p} = 1$, since then $\Cal O_{X_i,p}$ is a valuation ring.
Since $\uppi_i(p) = \mfrak$ as $\Cal O^*$ dominates $R$,
by construction we can take $q \in A_i$ such that $\Cal O_{X_i,p} \subseteq \Cal O_{X_i,q}$.
Then $1/\alpha \in \Cal O_{X_i,q} \cap K = \Cal O_{i,q}$, which is a contradiction as $\alpha \in \mfrak_{i,q}$.

Finally, suppose that $\Chara(R/\mfrak) = 2$, hence $\Chara(K) \ne 3$.
Follow the same argument as above, replacing $1 + X^4 = Y^2$ with $1 + X^3 = Y^3$, and apply the second case of Lemma~\ref{lem:discrete-val}.
\end{proof}

\begin{theorem}
\label{thm:resolution-2}
If $R$ is quasi-excellent local then the $\Sa E_K$-topology refines the $R$-adic topology.
\end{theorem}

\begin{proof}
By Fact~\ref{fact:ring-properties}(4) $R$ is $N$-1, so the integral closure $S$ of $R$ in $K$ is finite over $R$.
By Lemma~\ref{lem:N-1} it is enough to show that the $\Sa E_K$-topology refines the $S$-adic topology.
It is enough to show that $S$ is $\Sa E_K$-open.
Since $S$ is finite over the local ring $R$, $S$ has only finitely many maximal ideals $\mfrak_1,\ldots,\mfrak_k$.
Let $S_i$ be the localization of $S$ at $\mfrak_i$ for each $i \in \{1,\ldots,k\}$.
Then $S = S_1 \cap \ldots \cap S_k$, so it is enough to fix $i \in \{1,\ldots,k\}$ and show that $S_i$ is $\Sa E_K$-open.
Note that $S_i$ is a localization of a finite extension of the quasi-excellent ring $R$ and $S_i$ is a localization of the normal ring $S$.
By Fact~\ref{fact:ring-properties}
$S_i$ is quasi-excellent and normal.
Theorem~\ref{thm:resolutionByAlterations} (which applies by Theorem~\ref{thm:quasiExcellentResolutionByAlterations}) shows that $S_i$ is $\Sa E_K$-open.
\end{proof}

\begin{remark}
  In \cite[Theorem 3.35]{model-theoretic-algebra}, the henselian case of Fact \ref{fact:JL} is used to prove that any henselian regular local domain is first-order definable in its fraction field.
  Whether the same holds for a henselian quasi-excellent local domain $R$ remains open.

  We only obtain the weaker statement that the $R$-adic topology is definable in the fraction field $K$, i.e.\ there is a definable family of sets forming a basis for the $R$-adic topology:
  Indeed, we have shown in Lemma \ref{lem:discrete-val} that there is an étale image $\emptyset \neq U \subseteq K = \Aa^1_K(K)$ contained in $R$.
  Since $U$ is definable and open, the family $\{ aU + b \colon a \in K^\times, b \in K \}$ is a definable basis for the $R$-adic topology.

  Essentially the same argument shows that whenever $\Sa E_K$ is induced by a locally bounded field topology $\uptau$ (a situation which we shall study later in some detail), the topology $\uptau$ is definable.
\end{remark}

\section{Behaviour of \texorpdfstring{$\seak$}{E} under field extension}
\label{section:refinement}
\noindent
Suppose that $L/K$ is a finite field extension and let $[L:K] = d$.
We briefly describe the extension $\extlk(\seak)$ of the $\seak$-topology to $L$, see \cite[Section~4.5]{firstpaper} for details.
After fixing a $K$-basis for $L$ we may identify each $L^n$ with $K^{dn}$.
We declare the $\extlk(\Sa E_K)$-topology on $L^n$ to be the $\Sa E_K$-topology on $K^{dn}$.
This topology does not depend on the choice of the $K$-basis.
More generally, given a quasi-projective $L$-variety $V$ the $\extlk(\Sa E_K)$-topology on $V(L)$ is the $\Sa E_K$-topology on the $K$-points of the Weil restriction of $V$ (this set is canonically identified with $V(L)$.)
Any variety is Zariski-locally quasi-projective, so we can define the $\extlk(\seak)$-topology on the $K$-points of an arbitrary $K$-variety in a natural way.
For example $\operatorname{Ext}_{\Cc/\Rr}(\Sa E_\Rr)$ is the usual complex analytic topology over $\Cc$.

Endowing all $V(L)$ with the $\extlk(\Sa E_K)$-topology gives a well-behaved system of topologies in the sense of \cite[Definition 1.2]{firstpaper}, see the following consequence of \cite[Proposition-Definition 4.17]{firstpaper}.
\begin{fact}
\label{fact:extension-system}
Suppose that $L/K$ is finite and $V \to W$ is a morphism of $L$-varieties.
Then $V(L) \to W(L)$ is $\extlk(\seak)$-continuous.
In particular $L \to L, x \mapsto \alpha x + \beta$ is an $\extlk(\seak)$-homeomorphism for any $\alpha \in L^\times$, $\beta \in L$.
\end{fact}
\noindent
By \cite[Proposition~5.7]{firstpaper} $\extlk(\Sa E_K)$ refines $\Sa E_L$ for any finite $L/K$.
We would like to know when this refinement is strict.

Up to now we knew two examples.
If $K$ is real closed and $L = K(\sqrt{-1})$ then $\Sa E_L$ is the Zariski topology and $\seak$ is the order topology, hence $\extlk(\seak)$ strictly refines $\Sa E_L$.
Recall that the following are equivalent by \cite[Theorem~C.1]{firstpaper}:
\begin{enumerate}
\item $L$ is large,
\item the $\Sa E_L$-topology on $L$ is not discrete,
\item the $\Sa E_L$-topology on $V(L)$ is not discrete when $V$ is an $L$-variety with $V(L)$ infinite.
\end{enumerate}
By~\cite{Srinivasan} there are non-large fields with large finite extensions.
If $L$ is large and $K$ is not then the $\seak$-topology on $K^d$ is discrete, hence the $\extlk(\Sa E_K)$-topology on $L$ is discrete, and the $\Sa E_L$-topology on $L$ is not discrete.

We give a third example where $\extlk(\Sa E_K)$ strictly refines $\Sa E_L$.
This is also the first example where both $\Sa E_{K},\Sa E_{L}$ are non-discrete field topologies.

\begin{theorem}\label{thm:NonJapaneseExtTop}
  Let $R$ be a henselian regular local domain and $L$ a finite extension of the fraction field $K$ of $R$ such that the integral closure $S$ of $R$ in $L$ is not a finite $R$-module.
  Then $\extlk(\Sa E_K)$ strictly refines $\Sa E_L$.
\end{theorem}
\noindent
Note that any $R$ as in the theorem is by definition not Japanese and hence not quasi-excellent.
Since regular local rings are normal, Fact~\ref{fact:ring-properties}(7) shows that the theorem is only ever applicable in positive characteristic.

Before proving the theorem, we give an important special case in the language of valued fields.
See \cite[Example 3.5]{Kuhlmann_TheDefect} for an example of this situation.
\begin{corollary}
  Let $v$ be a henselian discrete valuation on a field $K$, $L/K$ a finite extension and $v'$ the unique prolongation of $v$ to $L$.
  If $(L,v')/(K,v)$ is a defect extension, i.e.\ $e(v'/v)f(v'/v) \lneq [L : K]$ where $e$ and $f$ are the relative ramification index and inertia degree, then $\extlk(\Sa E_K)$ strictly refines $\Sa E_L$.
\end{corollary}
\begin{proof}
  Let $R$ and $S$ be the valuation rings of $v$ and $v'$, respectively.
  Both are discrete valuation rings, in particular regular local domains.
  Furthermore, $S$ is the integral closure of $R$ in $L$ \cite[Chap.\ V, §8, no 3, Remarque]{Bourbaki_AlgebreComm567}, and the defect condition implies that $S$ is not a finite $R$-module \cite[Chap.\ V, §8, no 5, Théorème 2]{Bourbaki_AlgebreComm567}.
  Hence Theorem~\ref{thm:NonJapaneseExtTop} is applicable.
\end{proof}

\begin{proof}[Proof of Theorem \ref{thm:NonJapaneseExtTop}]
  Let $b_1, \dotsc, b_d \in L$ be a $K$-basis of $L$.
  By scaling with a suitable element of $R$, we may assume that for all $i$, $b_i \in S$.
  Hence $R' := R[b_1, \dotsc, b_d]$ is a finite $R$-module.
  The fraction field of $R'$ is $L$, and $S$ is the normalization of $R'$.
  By assumption, $S$ is not a finite $R'$-module, and thus by Lemma~\ref{lem:finite-module} the $R'$-adic topology on $L$ strictly refines the $S$-adic topology.
  The $S$-adic topology in turn refines $\Sa E_L$ (not necessarily strictly) by Fact~\ref{fact:intro-old-R-adic}(1), since $S$ is henselian local by Fact~\ref{fact:normalization}.

  Under the identification of $L$ with $K^d$ given by the basis $b_1, \dotsc, b_d$, the subgroup $b_1 R + \dotsb + b_d R \subseteq R' \subseteq L$ is identified with $R^d \subseteq K^d$, which is open in the product topology of $d$ copies of the $R$-adic topology on $K$.
  Since the $R$-adic topology on $K$ coincides with the $\Sa E_K$-topology by Fact~\ref{fact:intro-old-R-adic}, this means that $b_1 R + \dotsb + b_d R$ and hence $R'$ are $\extlk(\Sa E_K)$-open.
  By Fact \ref{fact:extension-system} it follows that $\extlk(\Sa E_K)$ refines the $R'$-adic topology on $L$, which strictly refines $\Sa E_L$.
\end{proof}

\section{A large collection of incomparable topologies on \texorpdfstring{$\Qq_p$}{Qp}}
\label{section:example}
\noindent
Fix a prime $p$.
In this section we produce $2^{2^{\aleph_0}}$-many henselian local subrings $R \subsetneq \Qq_p$ with fraction field $\Qq_p$ such that the corresponding $R$-adic topologies are pairwise incomparable.
This is interesting in light of Corollary \ref{cor:quasiexcellent-henselian}, which shows that this behaviour cannot occur for quasi-excellent $R$, since in this case the $R$-adic topology induces the étale open topology.
It is also in contrast to F.~K.~Schmidt's theorem \cite[Theorem 4.4.1]{EP-value}, which shows that any two henselian valuation rings on a field which is not separably closed induce the same topology (compare also Fact \ref{fact:intro-old-R-adic}(1, 2)).

Our approach is based on \cite[Section 5]{thirdpaper}.
Given a ($\Qq$-)derivation $\der\colon \Qq_p \to \Qq_p$ we let $E_\der$ be $\{ \alpha \in \Zz_p : \der \alpha \in \Zz_p\}$.
It is easy to see that $E_\der$ is a subring of $\Zz_p$.
Fact~\ref{fact:example} is a summary of the statements of \cite[Section 5]{thirdpaper}.

\begin{fact}
\label{fact:example}
If $\der$ is not identically zero then:
\begin{enumerate}
[leftmargin=*]
\item $E_\der$ is a one-dimensional Noetherian henselian local ring with fraction field $\Qq_p$, and
\item the $E_\der$-adic topology on $\Qq_p$ strictly refines the $p$-adic topology. 
\end{enumerate}
Furthermore $\widehat{E_\der}$ is isomorphic to $\Zz_p[X]/(X^2)$, hence $E_\der$ is not excellent.
\end{fact}

\noindent
Let $\Sa D$ be the set of derivations $\Qq_p \to \Qq_p$ which are not constant zero.
We say that $\der,\der^* \in \Sa D$ are constant multiples of each other if $\lambda \der = \der^*$ for some $\lambda \in \Qq_p$.
We prove:

\begin{theorem}
\label{thm:example}
If $\der,\der^* \in \Sa D$ are not constant multiples of each other, then the $E_\der$-adic topology does not refine the $E_{\der^*}$-adic topology and vice versa.
There is $I \subseteq \Sa D$ such that $|I| = 2^{2^{\aleph_0}}$ and if $\der,\der^* \in I$, $\der \ne \der^*$ then the $E_\der$-adic topology does not refine the $E_{\der^*}$-adic topology.
\end{theorem}
\noindent
Thus there are $2^{2^{\aleph_0}}$-distinct $E_\der$-adic topologies on $\Qq_p$.
We explain how the second claim follows from the first.
Let $B$ be a transcendence basis for $\Qq_p$.
By the usual rules for extending derivations to separable field extensions \cite[Section 2.8]{field-arithmetic}, it is easy to see that any function $B \to \Qq_p$ uniquely extends to a derivation $\Qq_p \to \Qq_p$.
Since $\lvert B \rvert = 2^{\aleph_0}$, this shows $\lvert\Sa D\rvert = 2^{2^{\aleph_0}}$.
As every element of $\Sa D$ is a constant multiple of precisely $\lvert \Qq_p^\times\rvert = 2^{\aleph_0}$ other elements of $\Sa D$, this shows that there are $2^{2^{\aleph_0}}$ classes of elements of $\Sa D$ under the equivalence relation of being a constant multiple of one another, and we may take $I \subseteq \Sa D$ to be a set of representatives for this equivalence relation.

\begin{lemma}
\label{lem:der}
Suppose that $\der, \der^* \colon \Qq_p \to \Qq_p$ are derivations and neither is a constant multiple of the other.
Then $\{ (\alpha, \der \alpha, \der^* \alpha) : \alpha \in \Qq_p\}$ is $p$-adically dense in $\Qq^3_p$.
\end{lemma}

\begin{proof}
As $\der,\der^*$ are not constant multiples of each other there are $s,t \in \Qq_p$ such that $(\der s, \der t)$ and $(\der^* s, \der^* t)$ are not scalar multiples of each other in $\Qq^2_p$.
Any derivation $\Qq_p \to \Qq_p$ is $\Qq$-linear and vanishes on $\Qq$, hence for $\alpha = a + sb + tc$ with $a, b, c \in \Qq$ we have
\begin{align*}
(\alpha,\der \alpha, \der^* \alpha) = (a + sb + tc,  b \der s + c \der t, b\der^* s + c \der^* t).
\end{align*}
We let $T \colon \Qq^3_p \to \Qq^3_p$ be the $\Qq$-linear transformation given as follows:
\[
T\begin{pmatrix} x \\ y \\z \end{pmatrix} = \begin{pmatrix} x + sy + tz \\  (\der s)y + (\der t)z \\ (\der^* s)y + (\der^* t)z \end{pmatrix}
\]
Note that $T(\Qq^3) \subseteq \{ (\alpha, \der \alpha, \der^* \alpha) : \alpha \in \Qq_p\}$, so it is enough to show that $T(\Qq^3)$ is dense in $\Qq^3_p$.
As $\Qq^3$ is dense in $\Qq^3_p$ and $T$ is linear it is sufficient to note that $T$ is invertible since
\[
\det(T) = \det \begin{pmatrix} 1 & s & t \\ 0 & \der s & \der t \\ 0 & \der^* s & \der^* t \end{pmatrix} = \det \begin{pmatrix}   \der s & \der t \\  \der^* s & \der^* t \end{pmatrix} \ne 0. \qedhere
\]
\end{proof}

\begin{proof}[Proof of Theorem~\ref{thm:example}]
Suppose $\der, \der^*$ are not constant multiples of each other.
We show that the $E_{\der^*}$-adic topology does not refine the $E_{\der}$-adic topology.
By Lemma~\ref{lem:adic-refinement} is enough to show that $a E_{\der^*} \nsubseteq E_{\der}$ for any $a \in \Qq^\times_p$.
Let 
\[U = \{(b,b',b'') \in \Zz_p \times \Qq_p \times \Zz_p : ab' + (\der a)b \in \Qq_p\setminus \Zz_p\}.\]
Then $U$ is open and nonempty as $(0,(pa)^{-1},0) \in U$.
By Lemma~\ref{lem:der} we have $(b, \der b, \der^* b) \in U$ for some $b \in \Qq_p$.
Then $b, \der^* b \in \Zz_p$, so $b \in E_{\der^*}$, and $\der(ab) = a (\der b) + b(\der a)  \notin \Zz_p$, so $ab \notin E_{\der}$.
\end{proof}

\section{The étale open topology on pseudo-algebraically closed fields}
\label{section:PAC}
\noindent
Recall that a field $K$ is \emph{pseudo-algebraically closed} (PAC) if every geometrically integral $K$-variety has a $K$-point.

\begin{proposition}\label{prop:pac}
  Let $K$ be a PAC field.
  Then the étale open topology on varieties over $K$ is not induced by a field topology on $K$.
\end{proposition}
\begin{proof}
  Suppose for a contradiction that there is a field topology $\uptau$ on $K$ inducing $\Sa E_K$.
  We consider the morphism $\alpha \colon \operatorname{PGL}_{2,K} \times \Pp^1_K \to \Pp^1_K$ given by the natural group action, as well as the projection morphisms $\pi_1 \colon \operatorname{PGL}_{2,K} \times \Pp^1_K \to \operatorname{PGL}_{2,K}$ and $\pi_2 \colon \operatorname{PGL}_{2,K} \times \Pp^1_K \to \Pp^1_K$.
  For later use we observe that the morphism $(\pi_1, \alpha) \colon \operatorname{PGL}_{2,K} \times \Pp^1_K \to \operatorname{PGL}_{2,K} \times \Pp^1_K$ is an isomorphism, since it has an obvious inverse given by acting with the inverse group element.
  In particular, the morphism $\alpha = \pi_2 \circ (\pi_1, \alpha)$ is smooth since $\pi_2$ is smooth (as it is a base change of the smooth morphism $\operatorname{PGL}_{2,K} \to \operatorname{Spec} K$).
  
  The étale open topology on $\operatorname{PGL}_{2,K}(K) \times \Pp^1_K(K)$ is the product topology of the étale open topologies on $\operatorname{PGL}_{2,K}(K)$ and $\Pp^1_K(K)$, since the analogous statement is true for the $\uptau$-topology and the two topologies agree on the $K$-points of every variety.
  Let $\emptyset \neq U \subseteq \Pp^1(K)$ be open.
  We show that $U$ is necessarily cofinite.

  The group scheme action $\alpha$ induces a map $\operatorname{PGL}_{2,K}(K) \times \Pp^1(K) \to \Pp^1(K)$ on $K$-points, which we also denote by $\alpha$.
  It is continuous by the defining properties of the étale open topology, and so there exist non-empty étale open subsets of $\operatorname{PGL}_{2,K}(K)$ and $\Pp^1(K)$ whose product is contained in the preimage of $U$ under $\alpha$.
  In other words, there exist two $K$-varieties $X$ and $Y$ with étale maps $X \to \operatorname{PGL}_{2,K}$, $Y \to \Pp^1_K$ such that $X(K), Y(K) \neq \emptyset$ and $U$ contains the image of $X(K) \times Y(K)$ under the composite
  \[ g \colon X \times Y \to \operatorname{PGL}_{2,K} \times \Pp^1_K \overset{\alpha}{\to} \Pp^1_K.\]
  The $K$-schemes $X$ and $Y$ are smooth since they are étale over the smooth $K$-schemes $\operatorname{PGL}_{2,K}$ and $\Pp^1_K$, respectively.
  Passing to a connected component of $X$ and $Y$ if necessary, we may additionally assume that both $X$ and $Y$ are connected (as schemes, i.e.\ not in relation to the topologies $\Sa E_K$ or $\uptau$).
  Since they both have a $K$-point, $X$ and $Y$ are then geometrically connected \cite[Proposition 2.3.24]{poonen-qpoints} and hence (by smoothness) geometrically integral \cite[Proposition 3.5.67]{poonen-qpoints}.

  We claim that the generic fibre of $g$ is geometrically integral (as a variety over the function field $K(\Pp^1_K)$), i.e.\ that the function field $K(X \times Y)$ is a regular extension of the function field $K(\Pp^1_K)$ via the map $g$.
  Let us defer the proof of this claim for the moment.
  By \cite[Tags 0578 and 0559]{stacks-project}, all but finitely many fibres of $g$ are geometrically integral.
  In particular, for all but finitely many $x \in \Pp^1(K)$, the $K$-variety $g^{-1}(x)$ has a $K$-point by the PAC property, and thus $x \in g(X(K) \times Y(K)) \subseteq U$.

  This shows that $U$ is cofinite.
  Thus the étale open topology on $K = \Aa^1_K(K) \subseteq \Pp^1_K(K)$, and therefore the topology $\uptau$, is the cofinite topology.
  Since the cofinite topology is not a field topology on any infinite field, this yields the desired contradiction.

  It remains to prove the claim.
  This is purely a matter of algebraic geometry, so the topologies $\Sa E_K$ and $\uptau$ no longer intervene.
  As a consequence of Zariski's Main Theorem, we can embed $X$ and $Y$ as open subschemes of normal integral schemes $\overline X$, $\overline Y$ with finite morphisms $p_1 \colon \overline X \to \operatorname{PGL}_{2,K}$, $p_2 \colon \overline Y \to \Pp^1_K$ extending the étale morphisms from $X$ respectively $Y$.
  (See for instance \cite[Theorem 3.5.52 (c)]{poonen-qpoints} (recalling that $X$ and $Y$ are separated by our convention on varieties), where $\overline X$ and $\overline Y$ are described concretely as normalisations of $\operatorname{PGL}_{2,K}$ (respectively $\Pp^1_K$) in the function field of $X$ (respectively $Y$).)

  Via the dominant morphism $\overline X \times \overline Y \overset{ p_1 \times p_2}{\longrightarrow} \operatorname{PGL}_{2,K} \times \Pp^1_K \overset{\alpha}{\to} \Pp^1_K$, which restricts to the morphism $g$ considered earlier on $X \times Y$, we can consider $K(X \times Y) = K(\overline X \times \overline Y)$ as an extension field of $K(\Pp^1_K)$.
  Let $F \subseteq K(X \times Y)$ be the relative algebraic closure of $K(\Pp^1_K)$ therein.
  Then $F/K$ is regular since $K(X \times Y)/K$ is regular, due to the geometric integrality of $X$ and $Y$.
  Let $C \to \Pp^1_K$ be the normalisation of $\Pp^1_K$ in $F$.
  Thus $C/K$ is a geometrically integral normal projective curve and $C \to \Pp^1_K$ is a finite morphism.
  We shall show using a ramification argument that in fact $C \to \Pp^1_K$ is an isomorphism.

  Let us consider the following diagram:
  \[ \xymatrix{
      \overline{X} \times \overline{Y} \ar^{p_1 \times p_2}[r] \ar@{-->}[d] & \operatorname{PGL}_{2,K} \times \Pp^1_K \ar^{(\pi_1, \alpha)}[d] \\
      \operatorname{PGL}_{2,K} \times C \ar[r] & \operatorname{PGL}_{2,K} \times \Pp^1_K
  }\]
All varieties occurring are geometrically integral and normal, the vertical morphism on the right is an isomorphism, the top horizontal morphism is finite and generically étale, and the bottom morphism (given by the identity on $\operatorname{PGL}_{2,K}$ and the previous map $C \to \Pp^1_K$) is finite.

  We can complete the diagram by a finite morphism on the left side, shown as a dashed arrow:
  Observe first that by construction, the function field $K(\overline X \times \overline Y)$ is an extension of the function field of $\operatorname{PGL}_{2,K} \times C$, i.e.~we can find a rational function on the left side making the diagram commute.
  In particular, we then have a normalisation of $\operatorname{PGL}_{2,K} \times C$ in the function field $K(\overline X \times \overline Y)$ (see for instance \cite[Definition 4.1.24]{Liu_AlgGeomAndArithCurves}), which is also a normalization of $\operatorname{PGL}_{2,K} \times \Pp^1_K$ in this field by construction.
  However, the morphism $(\pi_1, \alpha) \circ (p_1 \times p_2)$ already describes $\overline X \times \overline Y$ as the normalisation of $\operatorname{PGL}_{2,K} \times \Pp^1_K$ within $K(\overline X \times \overline Y)$;
  therefore, by uniqueness of normalisations, $\overline X \times \overline Y$ must already be the normalisation of $\operatorname{PGL}_{2,K} \times C$ in $K(\overline X \times \overline Y)$, and the morphism on the left side of the diagram making it commutative is none other but the normalisation morphism.

  Let us show that the morphism of curves $C \to \Pp^1_K$ is unramified.
  First observe that the only prime divisors of $\operatorname{PGL}_{2,K} \times \Pp^1_K$ which ramify under the map $p_1 \times p_2$ are of the form $D \times \Pp^1_K$ or $\operatorname{PGL}_{2,K} \times D'$, where $D$ ramifies under $p_1$ or $D'$ ramifies under $p_2$.
  Since $\alpha$ is a transitive group action, the image of such a prime divisor under the automorphism $(\pi_1, \alpha)$ of $\operatorname{PGL}_{2,K} \times \Pp^1_K$ is never of the form $\operatorname{PGL}_{2,K} \times \{ x \}$ for a closed point (i.e., prime divisor) $x$ of $\Pp^1_K$.
  In other words, for every closed point $x$ of $\Pp^1_K$, the prime divisor $\operatorname{PGL}_{2,K} \times \{ x \}$ does not ramify along the map $(\pi_1, \alpha) \circ (p_1 \times p_2) \colon \overline X \to \overline Y \to \operatorname{PGL}_{2,K} \times \Pp^1_K$.
  Due to the commutative diagram above, it follows that the prime divisor in question cannot ramify along $\operatorname{PGL}_{2,K} \times C \to \operatorname{PGL}_{2,K} \times \Pp^1_K$ either, and so $x$ is not a branch point of $C \times \Pp^1_K$.
  Since $x$ was arbitrary, this shows that $C \to \Pp^1_K$ is unramified.
  Since $C$ is a geometrically integral projective curve and $\Pp^1_K$ is geometrically simply connected (see \cite[Corollary 7.4.20]{Liu_AlgGeomAndArithCurves}), it follows that the map $C \to \Pp^1_K$ is an isomorphism, and so $F = K(C) = K(\Pp^1_K)$.
  In other words, the field $K(\Pp^1_K)$ is relatively algebraically closed in $K(X \times Y)$.
  
  Finally, the morphism $g$ is smooth,
  since it factors as the composition of the étale morphism $X \times Y \to \operatorname{PGL}_{2,K} \times \Pp^1_K$, and the smooth morphism $\alpha$.
  Smoothness of $g$ at the generic point means that $K(X \times Y)/K(\Pp^1_K)$ is a separable field extension, so (together with relative algebraic closedness) we have shown that it is a regular field extension.
  This finishes the proof of the claim that the generic fibre of $g$ is geometrically integral.
\end{proof}

\begin{remark}
  The precise choice of the morphism $\operatorname{PGL}_{2,K} \times \Pp^1_K \to \Pp^1_K$ in the proof above is not very important.
  We only used that it is a transitive group action on a geometrically simply connected variety.
  In characteristic zero, one can instead use the simpler addition action $\Aa^1_K \times \Aa^1_K \to \Aa^1_K$, but in positive characteristic $\Aa^1_K$ is not geometrically simply connected.
\end{remark}

\section{gt-henselian field topologies}
\label{section:t-henselian}

\subsection{Background on topological fields}
\label{section:field topology background}
We develop the basics of a theory of gt-henselian field topologies extending the Prestel-Ziegler theory of t-henselian field topologies.
Recall our convention that all field topologies are Hausdorff and non-discrete.
Throughout, we fix such a field topology $\uptau$ on the field $K$.

\begin{definition}\label{def:gt-henselian}
We say that $\uptau$ is \textbf{generalized (topologically) henselian}, for short \textbf{gt-henselian}, if for every $n$ and every neighbourhood $P \subseteq K$ of $-1$ there is a neighbourhood $O \subseteq K$ of zero such that the polynomial $X^{n + 1} + X^n + a_{n - 1} X^{n - 1} + \ldots + a_1 X + a_0$ has a root in $P$ for any $a_0,\ldots,a_{n - 1} \in O$.
\end{definition}

\noindent
As the terminology suggests, gt-henselianity generalizes t-henselianity. For more on t-henselianity, see~\cite[Section 7]{Prestel1978}.
\begin{remark}\label{rem:gt-henselian-literature}
  In fact, the field topology $\uptau$ is gt-henselian if and only if $K$ is $\uptau$-henselian in the sense considered in \cite[Examples 1.7]{Pop-little}, as follows from the characterization we give in Proposition \ref{prop:general}(\ref{item:3}) below.
  However, Pop's notion of $\uptau$-henselian rings does not seem to have been studied in any depth in the literature.
  We prefer the name gt-henselianity to stress the link with \cite{Prestel1978}.

  Another notion of henselianity for rings in the literature is given by the henselian semi-normed rings of \cite{FP-Galois} (on which Pop's definition of weak $\uptau$-henselianity is modelled), but there do not appear to be interesting examples of field topologies obtained in this way, except in the well-known case of a field with an absolute value.
\end{remark}

\noindent
Recall from \cite[Theorem 7.2 a)]{Prestel1978} (which we may as well take as a definition) that the field topology $\uptau$ on $K$ is \emph{t-henselian} if and only if it is a $V$-topology (see \cite[Section 3]{Prestel1978}) and for every $n \geq 1$ there is a $\uptau$-neighbourhood $U$ of $0$ such that any polynomial $f = X^n + X^{n-1} + a_{n-2} X^{n-2} + \dotsb + a_1 X + a_0 \in K[X]$ with $a_{n-2}, \dotsc, a_0 \in U$ has a zero in $K$.
We show that gt-henselianity generalises t-henselianity.
\begin{proposition}
  \label{prop:V-t}
  The topology $\uptau$ is $t$-henselian if and only if it is $gt$-henselian and a V-topology.
\end{proposition}

\noindent
For the proof we need the following fact, a special case of the  polynomial implicit function theorem for t-henselian fields~\cite[Theorem 7.4]{Prestel1978}.
We can also prove a polynomial implicit function theorem for locally bounded gt-henselian field topologies, but we will not do so here.

\begin{fact}
\label{fact:implicit-function}
Suppose that $\uptau$ is t-henselian, $f \in K[Y_1,\ldots,Y_n,X]$, and $(\alpha,\beta) \in K^n \times K$ is such that $f(\alpha,\beta) = 0 \ne \der f/\der X f(\alpha,\beta)$.
Then there are $\uptau$-neighbourhoods $U_1 \subseteq K^n$, $U_2 \subseteq K$ of $\alpha,\beta$, respectively, and a $\uptau$-continuous function $g \colon U_1 \to U_2$ such that 
\[
\{ (a,g(a)) : a \in U_1 \} = 
\{ (a,b) \in U_1 \times U_2 : f(a,b) = 0 \}
\]
\end{fact}

\begin{proof}[Proof of Proposition~\ref{prop:V-t}]
It follows directly from the definitions that a gt-henselian V-topology is t-henselian.
Suppose that $\uptau$ is t-henselian.
Then $\uptau$ is necessarily a V-topology.
Fix $n \ge 1$ and a neighbourhood $P \subseteq K$ of $-1$.
We let $f \in K[Y_1,\ldots,Y_{n-1},X]$ be the polynomial  $X^{n + 1} + X^n + Y_{n - 1}X^n + \ldots + Y_1 X + Y_0$.
Then we have $f(0,\ldots,0,-1) = 0 \ne \der f/\der X (0,\ldots,0,-1)$.
Let $U_1$, $U_2$, and $g$ be as in Fact~\ref{fact:implicit-function}.
Let $O = g^{-1}(P \cap U_2)$.
Then $O$ is a neighbourhood of zero.
By construction, if $\longpolycoeff O$ then $\longpoly$ has a root in $P$.
\end{proof}

\noindent
A significant set of examples for gt-henselian field topologies is furnished by $R$-adic topologies for $R$ henselian.
\begin{proposition}
  Let $R \subsetneq K$ be a henselian local domain with fraction field $K$.
  Then the $R$-adic topology is gt-henselian.
\end{proposition}
\begin{proof}
  Let $P \subseteq K$ be an $R$-adic neighbourhood of $-1$.
  Then $P$ contains $-1 + \alpha R$ for some $\alpha \in K^\times$.
  By multiplying $\alpha$ with a suitable element of $R$, we may assume that $\alpha \in R$ and $\alpha$ is not a unit.
  It now suffices to show that for every $n$ and all $a_0, \dotsc, a_{n-1} \in \alpha R$, the polynomial $X^{n+1} + X^n + a_{n-1}X^{n-1} + \dotsb + a_1 X + a_0$ has a root in $-1 + \alpha R$.
  This precisely means that $(R, \alpha R)$ is a \emph{henselian pair} (see the characterization in \cite[Tag 09XI (5)]{stacks-project}), which follows from \cite[Tag 0DYD]{stacks-project} since $(R, \mathfrak m)$ is a henselian pair (where $\mathfrak m$ is the maximal ideal of $R$).
\end{proof}

\noindent
We let $\polyn$ be the $K$-variety parameterizing degree $n$ monic polynomials, so $\polyn$ is just a copy of $\Aa^n$.
Recall that $\alpha \in K$ is a \textbf{simple root} of $f \in K[X]$ if $f(\alpha) = 0$ and $f'(\alpha) \ne 0$.

\begin{proposition}
\label{prop:general}
The following are equivalent:
\begin{enumerate}
[leftmargin=*]
\item\label{item:4} $\uptau$ is gt-henselian.
\item\label{item:6} For any $n$ and neighbourhood $P$ of $1$ there is a neighbourhood $O$ of $0$ such that if $a_0,\ldots,a_{n - 1} \in O$ then $X^{n+1} - X^n + a_{n - 1}X^{n - 1} + \ldots + a_1 X + a_0$ has a root in $P$.
\item\label{item:5} For any $n$ and neighbourhood $P$ of $-1$ there is a neighbourhood $O$ of $0$ such that if $c_2,\ldots,c_n \in O$ then $1 + X + c_2 X^2 + \ldots + c_n X^n$ has a root in $P$.
\item\label{item:3} If $\alpha \in K$ is a simple root of a monic polynomial $f \in K[X]$, $\deg f = n$, and $P \subseteq K$ is a neighbourhood of $\alpha$ then there is a neighbourhood $O \subseteq \polyn(K)$ of $f$ such that every $f^* \in O$ has a simple root in $P$.
\item\label{item:1} $V(K) \to W(K)$ is $\uptau$-open for any \'etale morphism $V \to W$.
\item\label{item:2} $V(K) \to W(K)$ is $\uptau$-open for any smooth morphism $V \to W$.
\end{enumerate}
\end{proposition}

\begin{definition}
\label{defn:standard etale}
A basic standard \'etale morphism is a morphism $\uppi \colon V \to W$ where $W$ is an affine $K$-variety, $V$ is the subvariety of $W \times \Aa^1$ given by $f = 0 \ne g$ for $f,g \in (K[W])[X]$ such that $f$ is monic, $\der f/\der X \ne 0$ on $V$, and $\uppi$ is the restriction of the projection $W \times \Aa^1 \to W$ to $V$.
A standard \'etale morphism is a morphism $\uppi\colon V \to W$ of $K$-varieties such that there is a $K$-variety isomorphism $\uprho \colon V^* \to V$ with $\uppi\circ\uprho\colon V^* \to W$ basic standard \'etale.
\end{definition}
\noindent
Fact~\ref{fact:locally standard} is \cite[Tag~02GT]{stacks-project}. 

\begin{fact}
\label{fact:locally standard}
Any \'etale morphism of $K$-varieties is locally standard \'etale.
That is, if $V \to W$ is an \'etale morphism of $K$-varieties and $p \in V$ then there is a Zariski open neighbourhood $V^* \subseteq V$ of $p$ and an affine Zariski open neighbourhood $W^* \subseteq W$ of $f(p)$ such that $f(W^*) \subseteq V^*$ and $V^* \to W^*$ is standard \'etale.
\end{fact}

\noindent
In the following proof, we work with respect to $\uptau$ throughout.

\begin{proof}[Proof of Proposition~\ref{prop:general}]
The equivalence of (\ref{item:4}) and (\ref{item:5}) is clear by considering the substitution $Y = 1/X$.
The equivalence of (\ref{item:4}) and (\ref{item:6}) is likewise clear by considering the substitution $Y = -X$.
The implication from (\ref{item:2}) to (\ref{item:1}) is clear since étale morphisms are smooth, and the converse holds since a smooth morphism is locally the composition of an étale morphism and a product projection \cite[Tag~054L]{stacks-project}, see also \cite[Proposition 3.1]{secondpaper}.

We show that (\ref{item:3}) implies (\ref{item:1}).
Suppose (\ref{item:3}) and let $\uppi\colon V \to W$ be \'etale.
We show that $V(K) \to W(K)$ is $\uptau$-open.
By Fact~\ref{fact:locally standard} we may suppose that $\uppi$ is basic standard \'etale.
Let $\uppi$, $f$, and $g$ be as in Definition~\ref{defn:standard etale}.
Given $\alpha \in W(K)$ let $f_\alpha \in K[X]$ be given by evaluating $f$ at $\alpha$ and let $\upiota \colon W(K) \to \polyn(K)$ be $\upiota(\alpha) = f_\alpha$.
Note that $\upiota$ is continuous with respect to $\uptau$.
It is enough to fix $(\alpha,\beta) \in V(K)$ and a neighbourhood $P \subseteq W(K) \times K$ of $(\alpha,\beta)$ and show that $\uppi(V(K) \cap P)$ is a neighbourhood of $\alpha$.
We may suppose that $P$ is contained in the open subvariety of $W \times \Aa^1$ given by $g \ne 0$.
As the $\uptau$-topology on $W(K) \times K$ is the product topology we suppose that $P = O^* \times U$ for a neighbourhood $O^* \subseteq W(K)$ of $\alpha$ and a neighbourhood $U \subseteq K$ of $\beta$.
Note that $\beta$ is a simple root of $f_\alpha$ as $\der f/\der X$ does not vanish at $(\alpha,\beta)$.
Hence there is a neighbourhood $O \subseteq \polyn(K)$ such that every $f^* \in O$ has a simple root in $U$.
We show that $O^* \cap \upiota^{-1}(O)$ is contained in $\uppi(V(K) \cap P)$, note that $O^* \cap \upiota^{-1}(O)$ is a neighbourhood of $\alpha$.
Fix $\gamma \in O^* \cap \upiota^{-1}(O)$.
Then $f_\gamma \in O$, hence $f_\gamma$ has a simple root $\eta$ in $U$.
We show that $(\gamma,\eta) \in V(K)\cap P$.
Note $f(\gamma,\eta) = f_\gamma(\eta) = 0$.
As $\gamma \in O^*$ and $\eta \in U$ we have $(\gamma,\eta) \in P$, so $g(\gamma, \eta) \ne 0$, hence $(\gamma,\eta) \in V(K)$.

We show that (\ref{item:1}) implies (\ref{item:3}).
Suppose (\ref{item:1}) and fix $n \ge 2$ ((\ref{item:3}) is trivial for $n=1$).
Let $V$ be the subvariety of $\Spec K[Y_1,\ldots,Y_n,X] = \Aa^{n } \times \Aa^1$ given by $X^n + Y_{n - 1} X^{n - 1} + \ldots + Y_1 X + Y_0 = 0$ and $(\der/\der X) [X^n + Y_{n - 1} X^{n - 1} + \ldots + Y_1 X + Y_0 ] \ne 0$.
Let $\uppi \colon V \to \Aa^n$ be the projection.
Then $\uppi$ is standard \'etale, hence the projection $V(K) \to K^n$ is open.
Suppose $a = (a_0,\ldots,a_{n - 1}) \in K^n$, $f(X) = X^n + a_{n - 1} X^{n - 1} + \ldots + a_1 X + a_0$,  $b \in K$ is a simple root of $f$, and $P \subseteq K$ is a neighbourhood of $b$.
Note that $(a,b) \in V(K)$.
Let $O = \uppi([K^n \times P] \cap V(K))$, so $O$ is a neighbourhood of $a$.
It is easy to see that $f^*(X) = X^n + a^*_{n - 1} X^{n - 1} + \ldots + a^*_1 X + a^*_0$ has a simple root in $P$ for any $(a^*_0,\ldots,a^*_{n - 1}) \in O$.

We show that (\ref{item:3}) implies (\ref{item:4}).
Let $P$ be a neighbourhood of $-1$.
Note that $-1$ is a simple root of $X^{n + 1} + X^n$.
Hence there is a neighbourhood $O \subseteq K^n$ of $(1,0,\ldots,0)$ such that if $a = (a_{n },\ldots,a_0) \in O$ then $X^{n + 1} + a_{n}X^{n} + \ldots + a_1 X + a_0$ has a root in $P$.
Fix a neighbourhood $Q \subseteq K$ of $0$ such that $\{1\} \times Q^n \subseteq O$.
Then $X^{n + 1} + X^n + a_{n - 1} X^{n - 1} + \ldots + a_1 X + a_0$ has a root in $P$ for all $a_0,\ldots,a_{n - 1} \in Q$.
Hence $\uptau$ is gt-henselian.

We finish by showing that (\ref{item:5}) implies (\ref{item:3}).
Suppose (\ref{item:5}).
Let $f \in K[X]$ be monic of degree $n$, and let $\alpha \in K$ be a simple root of $f$.
The change of variables $Y = X - \alpha$ induces an automorphism of $\polyn(K)$, so we may assume without loss of generality that $\alpha = 0$.
Thus $f = a_1 X + \dotsb + a_{n-1}X^{n-1} + X^n$ with coefficients $a_i \in K$, $a_1 = f'(0) \neq 0$.

Let $P \subseteq K$ be a neighbourhood of $0$.
Let $P' \subseteq P$ be a smaller neighbourhood of $0$ such that $-1 \not\in P'$ and $a_1^{-1} \cdot P' \cdot (1 + P')^{-1}(-1+P') \subseteq P$.
By (\ref{item:5}) there exists a neighbourhood $O$ of $0$ such that every polynomial $1 + X + c_2 X^2 + \dotsb + c_n X^n$ with $c_i \in O$ has a root in $-1+P'$.
By shrinking $O$ and $P'$, we may assume that any root in $-1+P'$ of any such polynomial is simple.

Let $O^\ast$ be the set of polynomials $b_0 + b_1 X + \dotsb + b_{n-1}X^{n-1} + X^n$ in $\polyn(K)$ with $b_0 \in P'$, $b_1 \in a_1 (1 + P')$, and $b_i b_0^{i-1} b_1^{-i} \in O$ for all $i = 2, \dotsc, n$.
This is a neighbourhood of $f$ in $\polyn(K)$.

Let us show that every $g = b_0 + b_1 X + \dotsb + b_{n-1}X^{n-1} + X^n \in O^\ast$ has a simple root in $P$.
If $b_0 = 0$, then $0$ is a simple root of $g$.
Otherwise, consider the polynomial $h = b_0^{-1}g(b_0 b_1^{-1}X) \in K[X]$.
By construction, $h$ has the form $1 +X + c_2 X^2 + \dotsb + c_n X^n$ with $c_i \in O$, and thus has a simple zero in $-1+P'$.
Therefore $g$ has a simple zero in $b_0 b_1^{-1}(-1+P') \subseteq P$, as desired.
\end{proof}

\begin{remark}
  We have seen in Section \ref{section:example} that the field $\Qq_p$ carries $2^{2^{\aleph_0}}$ many pairwise incomparable locally bounded gt-henselian topologies.
  This is in marked contrast to t-henselian topologies, where a field which is not separably closed can admit at most one such (\cite[Theorem 7.9]{Prestel1978}, essentially F.~K.~Schmidt's theorem on independent henselian valuations).
  Therefore, while it is sensible to speak of t-henselian fields and \emph{the} t-henselian topology on one such (forbidding separably closed fields), we avoid the analogous terminology in the gt-henselian case.
\end{remark}

\noindent
The analysis of the topological field $(K, \uptau)$ simplifies when $\uptau$ is $\omega$-complete, i.e.\ it the collection of neighbourhoods of $0$ is closed under countable intersections.
Using an ultrapower argument, Prestel-Ziegler in \cite[Theorem 1.1]{Prestel1978} show that every $(K, \uptau)$ may be replaced by some $(K^\ast, \uptau^\ast)$ which is ``locally equivalent'' to $(K, \uptau)$ and such that $\uptau^\ast$ is $\omega$-complete.
Here local equivalence means that $(K, \uptau)$ and $(K^\ast, \uptau^\ast)$ satisfy the same sentences in a certain logic extending first-order logic in the language of rings, allowing restricted second-order quantification over neighbourhoods of $0$.
See \cite[Section 1]{Prestel1978} for details on this formalism.
\begin{lemma}\label{lem:transfer}
  Let $(K^\ast, \uptau^\ast)$ be locally equivalent to $(K, \uptau)$.
  Then $\uptau^\ast$ is gt-henselian (t-henselian) if and only if $\uptau$ is gt-henselian (t-henselian).
\end{lemma}
\begin{proof}
  It is immediate from the definition that gt-henselianity is expressed by a collection of local sentences.
  The same holds for t-henselianity (as is already expressed in \cite[Corollary 7.3]{Prestel1978}).
\end{proof}

\noindent
For $\omega$-complete field topologies, we have the following.
\begin{fact}
\label{fact:pz}
Suppose that $\uptau$ is $\omega$-complete.
Then $\uptau$ is locally bounded if and only if $\uptau$ is the $S$-adic topology for a local subring $S$ of $K$ with $K = \Frac(S)$.
Furthermore $\uptau$ is a V-topology if and only if $\uptau$ is the $S$-adic topology for a valuation subring $S$ of $K$ and $\uptau$ is t-henselian if and only if $\uptau$ is the $S$-adic topology for a henselian valuation subring $S$ of $K$.
\end{fact}

\noindent
We note that Fact~\ref{fact:pz} can fail without $\omega$-completeness.
For instance, it fails for the usual topology on $\Rr$ or $\Cc$.

\begin{proof}[Proof of Fact~\ref{fact:pz}]
The first claim is in the proof of \cite[Theorem~2.2 (b)]{Prestel1978}, the second is~\cite[Lemma~3.3]{Prestel1978}, and the third follows from \cite[Theorem~7.2]{Prestel1978}.
\end{proof}

\noindent
A subset of $K$ is a \textbf{henselian ideal} if it is the maximal ideal of a henselian local subring of $K$ with fraction field $K$.
We say that $\uptau$ is induced by a henselian local ring if $\uptau$ is the $R$-adic topology for a henselian local subring~$R$~of~$K$ with $\Frac(R)=K$.

\begin{proposition}
\label{prop:t-hensel}
Suppose $\uptau$ is $\omega$-complete.
The following are equivalent:
\begin{enumerate}
\item $\uptau$ is gt-henselian.
\item $\uptau$ admits a neighbourhood basis at zero consisting of henselian ideals.
\end{enumerate}
If $\uptau$ is also locally bounded then $\uptau$ is gt-henselian if and only if $\uptau$ is induced by a henselian local ring.
\end{proposition}

\noindent
Hence an $\omega$-complete gt-henselian field topology is a union of henselian field topologies.
For the proof of the proposition, we partly follow the proof of \cite[Theorem 2.2]{Prestel1978}, see also \cite[Theorem 7.2]{Prestel1978}.

\begin{proof}
Let us first assume that (2) holds and show that this implies (1).
Any $\uptau$-neighbourhood $P \subseteq K$ of $-1$ contains a set $-1 + I$, where $I$ is a henselian ideal which is a $\tau$-neighbourhood of $0$.
Applying condition (3) from Fact~\ref{fact:hensel-equiv}, we see that $O = I$ satisfies the condition from Definition~\ref{def:gt-henselian}.
Thus $\uptau$ is gt-henselian.

For the converse direction,
let us suppose that $\uptau$ is gt-henselian.
We wish to show that (2) holds.
We fix a neighbourhood $Q$ of zero and construct an open henselian ideal $P \subseteq K$ which is contained in $Q$.
We use Fact~\ref{fact:hensel-equiv} to show that $P$ is a henselian ideal.
Let $K_{\mathrm{pr}}$ be the prime subfield~of~$K$.
\begin{claim-star}
Suppose that $O$ is a neighbourhood of $-1$.
Then there is a neighbourhood $P \subseteq Q$ of zero such that:
\begin{enumerate}
\item $X^{n + 1} + X^n + a_{n - 1} X^{n - 1} + \ldots + a_1 X +  a_0$ has a root in $O$ when $a_0,\ldots,a_{n - 1} \in P$,
\item $K_\mathrm{pr} + P$ is a local subring of $K$ with fraction field $K$ and maximal ideal $P$.
\end{enumerate}
\end{claim-star}

\begin{claimproof}
By gt-henselianity, for every $n \ge 2$ we may fix a neighbourhood $U_n$ of $0$ such that if $a_0,\ldots,a_{n - 1} \in U_n$ then $X^{n + 1} + X^n + a_{n - 1} X^{n - 1} + \ldots + a_1 X + a_0$ has a root in $K$.
By $\omega$-completeness there is a neighbourhood $U$ of zero such that $U \subseteq U_n$ for all $n$.
We may suppose that $U \subseteq Q$ and that $U$ does not contain $1$.
Let $r_1,r_2,\ldots$ be an enumeration of $K_\mathrm{pr}$.
Construct a descending sequence $(P_i : i \in \Nn)$ of open neighbourhoods of zero such that $P_0 = U$, and for all $i \ge 1$ the sets $P_i + P_i$, $P_i - P_i$, $P_i \cdot P_i$, and $r_1 P_i, \ldots, r_i P_i$ are all contained in $P_{i - 1}$ and $(1 + P_i)^{-1} \subseteq 1 + P_{i - 1}$.
By $\omega$-completeness $P := \bigcap_{i \in \Nn} P_i$ is a neighbourhood of zero.
The proof of \cite[Theorem~2.2]{Prestel1978} shows that $K_\mathrm{pr} + P$ is a local subring of $K$ with maximal ideal $P$.
Finally, $K_\mathrm{pr} + P$ is open so Fact~\ref{fact:open subring} shows that $K = \Frac(K_\mathrm{pr} + P)$.
\end{claimproof}
\noindent
Inductively construct sequences $(P_i : i \in \Nn)$, $(O_i:i \in \Nn)$ of open neighbourhoods of $0$, $-1$, respectively such that $P_0 \subseteq Q$, $O_0 \subseteq O$ and for each $i \in \Nn$ and $n \ge 1$:
\begin{enumerate}
\item $X^{n + 1} + X^n + a_{n - 1} X^{n - 1} + \ldots + a_1 X + a_0$ has a root in $O_i$ for any $a_0,\ldots,a_{n - 1} \in P_i$,
\item $K_{\mathrm{pr}} + P_i$ is a local ring with maximal ideal $P_i$ and fraction field $K$, 
\item $P_i \subseteq O_i - 1$ and $O_{i + 1} \subseteq P_i + 1$.
\end{enumerate}
Let $P := \bigcap_{i \in \Nn} P_i$ and $O := \bigcap_{i \in \Nn} O_i$.
Note that $O = P + 1$.
By $\omega$-completeness $P$ is a neighbourhood of $0$.
Let $R = K_\mathrm{pr} + P$.
Since $K_{\mathrm{pr}} + P_i$ is a local ring with maximal ideal $P_i$ for each $i$, we easily check that $R$ is a ring and $P$ is an ideal with residue field $K_{\mathrm{pr}}$.
Furthermore, every element of $1 + P$ is invertible in $R$, since we have $(1+P_i)^{-1} \subseteq 1 + P_i \subseteq R$ for every $i$.
It follows that $R$ is a local ring with maximal ideal $P$.
As $P$ is a neighbourhood of zero, $R$ is open.
By Fact~\ref{fact:open subring} $K = \Frac(K_\mathrm{pr} + P)$.
Note that $X^{n + 1} + X^n + a_{n - 1} X^{n - 1} + \ldots + a_1 X + a_0$ has a root in $P + 1$ for every $a_0,\ldots,a_{n - 1} \in P$.
Hence $R$ is henselian by Fact~\ref{fact:hensel-equiv}.

We now suppose that $\uptau$ is locally bounded.
We take $Q$ in the construction above to be bounded, hence $P$ is bounded.
Then $\alpha + P$ is bounded for all $\alpha \in K_\mathrm{pr}$, so $K_\mathrm{pr} + P$ is bounded as a countable union of bounded sets, since $(K, \uptau)$ is $\omega$-complete.

An application of Fact~\ref{fact:open subring 2} shows that $\uptau$ is the $R$-adic topology.
\end{proof}

\begin{corollary}
\label{cor:azumayan0}
Suppose that $\uptau$ is locally bounded.
Let $(K^*,\uptau^*)$ be locally equivalent to $(K,\uptau)$ and $\omega$-complete.
Then $\uptau$ is gt-henselian if and only if $\uptau^*$ is induced by a henselian local ring.
\end{corollary}

\begin{proof}
Suppose $\uptau$ is gt-henselian.
By Lemma~\ref{lem:transfer} $\uptau^*$ is gt-henselian and by Proposition~\ref{prop:t-hensel} $\uptau^*$ is induced by a henselian local ring.
Conversely, if $\uptau^*$ is induced by a henselian local ring, then $\uptau^*$ is gt-henselian, hence $\uptau$ is gt-henselian by Lemma~\ref{lem:transfer}.
\end{proof}

\subsection{When is the \'etale open topology induced by a locally bounded field topology?}
\label{section:when is}

\begin{lemma}
\label{lem:azu-refine}
If $\uptau$ is gt-henselian then $\uptau$ refines the \'etale open topology.
If $\uptau$ induces the \'etale open topology then $\uptau$ is gt-henselian.
\end{lemma}

\begin{proof}
Suppose $\uptau$ is gt-henselian and $V$ is a $K$-variety.
By Proposition~\ref{prop:general}, any étale image in $V(K)$ is $\uptau$-open, and so $\uptau$ refines the $\seak$-topology on $V(K)$.
If $\uptau$ induces the \'etale open topology then Fact~\ref{fact:basic}(2) shows that if $V \to W$ is an \'etale morphism of $K$-varieties then $V(K) \to W(K)$ is $\uptau$-open.
Now again apply Proposition~\ref{prop:general}.
\end{proof}

\noindent
Corollary~\ref{cor:large} follows from Lemma~\ref{lem:azu-refine} and the fact that $K$ is large if and only if the \'etale open topology on $K$ is not discrete~\cite[Theorem~C]{firstpaper}.

\begin{corollary}
\label{cor:large}
If $K$ admits a gt-henselian field topology then $K$ is large.
\end{corollary}

\noindent
See \cite{Pop-little} for a definition and an account of largeness.
Corollary~\ref{cor:large} generalizes the theorem that the fraction field of a henselian local domain is large~\cite{Pop-henselian}.
Corollary~\ref{cor:large} was also observed in slightly more general form in~\cite[Theorem 1.8]{Pop-little}.

\begin{proposition}
\label{prop:azumayan}
Suppose that $\uptau$ is locally bounded.
Then the following are equivalent:
\begin{enumerate}
\item $\uptau$ induces the \'etale open topology over $K$.
\item $\uptau$ is gt-henselian and some nonempty \'etale image in $K$ is bounded.
\end{enumerate}
In particular the following are equivalent when $R$ is a local domain with fraction field $K$:
\begin{enumerate}
\setcounter{enumi}{2}
\item The $R$-adic topology agrees with the \'etale open topology.
\item The $R$-adic topology is gt-henselian and $R$ contains a nonempty \'etale image.
\end{enumerate}
\end{proposition}

\begin{proof}
The second equivalence follows easily from the first and the definitions.
We prove the first equivalence.
Suppose (1).
Then $\uptau$ is gt-henselian by Lemma~\ref{lem:azu-refine}.
As $\uptau$ is locally bounded we may fix a bounded open $U \subseteq K$.
Then $U$ contains a nonempty \'etale image, which is also bounded.
Thus (2) holds.
Now suppose (2).
By Lemma~\ref{lem:azu-refine} and Fact~\ref{fact:refine} it suffices to show that the $\seak$-topology on $K$ refines $\uptau$. 
Let $f \colon V \to \Aa^1$ be an \'etale morphism of $K$-varieties such that $U = f(V(K))$ is bounded.
By Proposition~\ref{prop:general} $U$ is $\uptau$-open, hence $(\alpha U + \beta : \alpha \in K^\times, \beta \in K)$ is a basis for $\uptau$.
Fact~\ref{fact:basic}(3) shows that each $\alpha U + \beta$ is $\seak$-open.
\end{proof}

\begin{corollary}
\label{cor:existential}
Suppose that $K$ is perfect and $\uptau$ is a locally bounded field topology on $K$.
Then the following are equivalent:
\begin{enumerate}
\item $\uptau$ induces the \'etale open topology over $K$.
\item $\uptau$ is gt-henselian and $f(V(K))$ is bounded for some $K$-variety morphism $f \colon V \to \Aa^1$ with $f(V(K))$ infinite.
\end{enumerate}
In particular the following are equivalent when $R$ is a local domain with fraction field $K$:
\begin{enumerate}
\item The $R$-adic topology agrees with the \'etale open topology.
\item The $R$-adic topology is gt-henselian and $f(V(K)) \subseteq R$ for some $K$-variety morphism $f \colon V \to \Aa^1$ with $f(V(K))$ infinite.
\end{enumerate}
\end{corollary}
\begin{proof}
  This follows from Proposition~\ref{prop:azumayan} and the fact that if $K$ is perfect and $f \colon V \to \Aa^1$ is a $K$-variety morphism then $f(V(K))$ is the union of a definable $\seak$-open set and a finite set, see~\cite[Theorem B]{secondpaper}.
\end{proof}

\begin{lemma}\label{lem:elem-equiv-EK-loc-bdd}
Suppose $K \equiv K'$ and $\Sa E_K$ is induced by some locally bounded field topology $\uptau$.
Then $\Sa E_{K'}$ is induced by a locally bounded field topology $\uptau'$.
Furthermore, $(K, \uptau)$ is locally equivalent to $(K', \uptau')$.
\end{lemma}
\begin{proof}
  We first argue that $\cE_{K'}$ on $K'=\Aa^1_{K'}(K')$ is a locally bounded field topology.
  For every $d > 0$, let
  \begin{align*} \mathcal{U}_d := \{ & U \subseteq K \colon U = \{ x \in K \colon \exists y \in K, f(x,y) = 0, g(x,y) \neq 0, \frac{\partial f}{\partial Y}(x,y) \neq 0 \} \\ & \text{ for some $f,g \in K[X,Y]$ of total degree $\leq d$} \}. \end{align*}
  Every $U \in \mathcal{U}_d$ is an étale open subset of $K = \Aa^1_K(K)$, and by Fact \ref{fact:locally standard} every étale open subset of $K$ is a union of elements of $\bigcup_d \mathcal{U}_d$, so that the étale open topology on $K = \Aa^1_K(K)$ has $\bigcup_d \mathcal{U}_d$ as a basis.
  Let $\mathcal{V}_d$ be the collection of $V \in \mathcal{U}_d$ such that $V \neq \emptyset$ and every $U \in \mathcal{U}_d$ is a union of scaled translates of $V$, i.e.\ sets of the form $a V + b$ with $a \in K^\times$, $b \in K$.

  There exists some $\uptau$-bounded étale image $0 \in U \subseteq K$, and by possibly shrinking $U$ we may assume that $U \in \mathcal{U}_d$ for some $d$ which we now fix.
  The scaled translates of $U$ form a basis for the étale open topology on $K = \Aa^1_K(K)$, i.e.\ a basis of $\uptau$.
  In particular $U \in \mathcal{V}_d$, so $\mathcal{V}_d$ is not empty.
  Note that by definition for any other element $V \in \mathcal{V}_d$ we may write $U$ as a union of scaled translates of $V$.
  
  Let $\mathcal{U}'_d, \mathcal{V}'_d$ be the collection of subsets of $K'$ defined analogously as $\mathcal{U}_d$ and $\mathcal{V}_d$ in $K$.
  Since these are definable families and $K' \equiv K$, there exists $U' \in \mathcal{V}'_d$. Moreover, we may assume the scaled translates of $U'$ form a basis of a field topology $\uptau'$ on $K'$.

  Since the families $\mathcal{V}_d$ and $\mathcal{V}'_d$ of subsets of $K$ resp.\ $K'$ are defined by the same (parameter-free) formula, and are bases of $\uptau$ resp.\ $\uptau'$, we have local equivalence of $(K, \uptau)$ and $(K', \uptau')$. In particular, $(K',\uptau')$ is gt-henselian. By Proposition~\ref{prop:azumayan}, $\uptau'$ induces the \'etale open topology on $K'$.
\end{proof}

\noindent
We now want to analyse the situation of a locally bounded $\omega$-complete topology.

\begin{theorem}
\label{thm:azumayan agreement}
Suppose that $\uptau$ is locally bounded and induces the $\seak$-topology.
Let $(K^*,\uptau^*)$ be locally equivalent to $(K,\uptau)$ and $\omega$-complete.
Then $\uptau^*$ is induced by a henselian local ring and $\uptau^*$ induces the \'etale open topology~over~$K^*$.
\end{theorem}
\begin{proof}
  By Lemma~\ref{lem:elem-equiv-EK-loc-bdd}, there is a locally bounded field topology $\uptau'$ on $K^\ast$ inducing $\Sa E_{K^\ast}$.
  Since $(K^\ast, \uptau^\ast)$ is gt-henselian by local equivalence, $\uptau^\ast$ refines $\uptau'$.
  On the other hand, there exists an étale image in $K'$ which is $\uptau^\ast$-bounded, since the same holds in $(K, \uptau)$.
  Thus $\uptau^\ast = \uptau'$ by \cite[Lemma 2.1(f)]{Prestel1978}.
  The statement now follows from Corollary~\ref{cor:azumayan0}.
\end{proof}

\noindent
Theorem~\ref{thm:azumayan agreement} reduces the question ``When is the \'etale open topology induced by a locally bounded field topology?" to the question ``When does the \'etale open topology agree with the $R$-adic topology for a henselian local ring $R$?".

\begin{proposition}
\label{prop:saturated}
Suppose that $K$ is $\aleph_1$-saturated and suppose that the \'etale open topology over $K$ is induced by a locally bounded field topology on $K$.
Then there is a henselian local subring $R$ of $K$ such that the \'etale open topology over $K$ agrees with the $R$-adic topology. 
\end{proposition}

\noindent
A field is $\aleph_1$-saturated if any descending sequence of nonempty definable sets has nonempty intersection.
Such fields can for instance be produced using the ultrapower construction, see \cite[Theorem 6.1.1]{ChangKeisler}.

\begin{proof}
  Let $\uptau$ be the locally bounded field topology inducing the \'etale open topology over $K$.
By Theorem~\ref{thm:azumayan agreement} it is enough to show that $\uptau$ is $\omega$-complete.
Fix a $\uptau$-bounded \'etale image $U$ in $K$ which contains $0$.
Then $\mathcal{B} = (\alpha U : \alpha \in K^\times)$ forms a neighbourhood basis for $\uptau$ at zero consisting of definable sets, see \cite[Lemma 2.1 (e)]{Prestel1978}.
By $\aleph_1$-saturation any intersection of countably many elements of $\mathcal{B}$ contains an element of $\mathcal{B}$.
Hence $\uptau$ is $\omega$-complete.
\end{proof}

\noindent
Theorem \ref{thm:intro-loc-bdd-henselian-R-adic} from the introduction follows from the preceding proposition together with Lemma \ref{lem:elem-equiv-EK-loc-bdd}.

\bibliographystyle{amsalpha}
\bibliography{ref}

\providecommand{\bysame}{\leavevmode\hbox to3em{\hrulefill}\thinspace}
\providecommand{\MR}{\relax\ifhmode\unskip\space\fi MR }
\providecommand{\MRhref}[2]{%
  \href{http://www.ams.org/mathscinet-getitem?mr=#1}{#2}
}
\providecommand{\href}[2]{#2}
\begin{thebibliography}{JTWY22}

\bibitem[BCR98]{real-algebraic-geometry}
J.~Bochnak, M.~Coste, and M-F. Roy, \emph{Real algebraic geometry}, Springer,
  1998.

\bibitem[Bou06]{Bourbaki_AlgebreComm567}
N.~Bourbaki, \emph{Éléments de mathématique. {A}lgèbre commutative.
  {C}hapitres 5 à 7}, Springer, 2006, reprint of the 1975 original.

\bibitem[CK90]{ChangKeisler}
C.~C. Chang and H.~J. Keisler, \emph{Model theory}, third ed., Studies in Logic
  and the Foundations of Mathematics, vol.~73, North-Holland Publishing Co.,
  Amsterdam, 1990. \MR{1059055}

\bibitem[Cut04]{cutkosky}
Steven~Dale Cutkosky, \emph{Resolution of singularities}, Graduate Studies in
  Mathematics, vol.~63, American Mathematical Society, Providence, RI, 2004.
  \MR{2058431}

\bibitem[DF21]{DittmannFehm-nondefinability}
Philip Dittmann and Arno Fehm, \emph{Nondefinability of rings of integers in
  most algebraic fields}, Notre Dame Journal of Formal Logic \textbf{62}
  (2021), no.~3, 589--592.

\bibitem[Duc09]{ducros-excellent}
Antoine Ducros, \emph{Les espaces de {B}erkovich sont excellents}, Ann. Inst.
  Fourier (Grenoble) \textbf{59} (2009), no.~4, 1443--1552. \MR{2566967}

\bibitem[EP05]{EP-value}
Antonio~J. Engler and Alexander Prestel, \emph{Valued fields}, Springer
  Monographs in Mathematics, Springer-Verlag, Berlin, 2005. \MR{2183496}

\bibitem[FJ05]{field-arithmetic}
Michael~D. Fried and Moshe Jarden, \emph{Field arithmetic}, Springer Berlin
  Heidelberg, 2005.

\bibitem[FP11]{FP-Galois}
Arno Fehm and Elad Paran, \emph{Galois theory over rings of arithmetic power
  series}, Adv. Math. \textbf{226} (2011), no.~5, 4183--4197. \MR{2770445}

\bibitem[Gab92]{Gabber-K-theory}
Ofer Gabber, \emph{{$K$}-theory of {H}enselian local rings and {H}enselian
  pairs}, Algebraic {$K$}-theory, commutative algebra, and algebraic geometry
  ({S}anta {M}argherita {L}igure, 1989), Contemp. Math., vol. 126, Amer. Math.
  Soc., Providence, RI, 1992, pp.~59--70. \MR{1156502}

\bibitem[Gro65]{EGAIV-2}
A.~Grothendieck, \emph{\'{E}l\'{e}ments de g\'{e}om\'{e}trie alg\'{e}brique.
  {IV}. \'{E}tude locale des sch\'{e}mas et des morphismes de sch\'{e}mas.
  {II}}, Inst. Hautes \'{E}tudes Sci. Publ. Math. (1965), no.~24, 231.
  \MR{199181}

\bibitem[Gro67]{EGA-IV-4}
\bysame, \emph{\'{E}l\'{e}ments de g\'{e}om\'{e}trie alg\'{e}brique. {IV}.
  \'{E}tude locale des sch\'{e}mas et des morphismes de sch\'{e}mas. {IV}},
  Inst. Hautes \'{E}tudes Sci. Publ. Math. (1967), no.~32, 361. \MR{238860}

\bibitem[Hir64]{Hironaka}
Heisuke Hironaka, \emph{Resolution of singularities of an algebraic variety
  over a field of characteristic zero. {I}, {II}}, Ann. of Math. (2) {\bf 79}
  (1964), 109--203; ibid. (2) \textbf{79} (1964), 205--326. \MR{0199184}

\bibitem[HRW04]{heinzer-rotthaus-wiegand-catenary-local}
William Heinzer, Christel Rotthaus, and Sylvia Wiegand, \emph{Catenary local
  rings with geometrically normal formal fibers}, Algebra, arithmetic and
  geometry with applications ({W}est {L}afayette, {IN}, 2000), Springer,
  Berlin, 2004, pp.~497--510. \MR{2037106}

\bibitem[ILO14]{TravauxDeGabber}
Luc Illusie, Yves Laszlo, and Fabrice Orgogozo (eds.), \emph{Travaux de
  {G}abber sur l'uniformisation locale et la cohomologie \'{e}tale des
  sch\'{e}mas quasi-excellents}, Soci\'{e}t\'{e} Math\'{e}matique de France,
  Paris, 2014, S\'{e}minaire \`a l'\'{E}cole Polytechnique 2006--2008. [Seminar
  of the Polytechnic School 2006--2008], With the collaboration of
  Fr\'{e}d\'{e}ric D\'{e}glise, Alban Moreau, Vincent Pilloni, Michel Raynaud,
  Jo\"{e}l Riou, Beno\^{\i}t Stroh, Michael Temkin and Weizhe Zheng,
  Ast\'{e}risque No. 363-364 (2014) (2014). \MR{3309086}

\bibitem[JL89]{model-theoretic-algebra}
Christian~U. Jensen and Helmut Lenzing, \emph{Model-theoretic algebra with
  particular emphasis on fields, rings, modules}, Algebra, Logic and
  Applications, vol.~2, Gordon and Breach Science Publishers, New York, 1989.
  \MR{1057608}

\bibitem[Joh20]{dp-finite-v}
Will Johnson, \emph{Dp-finite fields v: topological fields of finite weight},
  arXiv preprint arXiv:2004.14732 (2020).

\bibitem[JTWY22]{firstpaper}
Will Johnson, Minh~Chieu Tran, Erik Walsberg, and Jinhe Ye, \emph{\'{E}tale
  open topology and the stable field conjecture}, arXiv:2009.02319, accepted in
  J. Eur. Math. Soc. (2022+).

\bibitem[JWY21]{thirdpaper}
Will Johnson, Erik Walsberg, and Jinhe Ye, \emph{The \'{e}tale open topology
  over the fraction field of a henselian local domain}, arXiv preprint
  arXiv:2108.01868, accepted in Math. Nachr. (2021).

\bibitem[Kuh11]{Kuhlmann_TheDefect}
Franz-Viktor Kuhlmann, \emph{The defect}, Commutative algebra---{N}oetherian
  and non-{N}oetherian perspectives, Springer, New York, 2011, pp.~277--318.
  \MR{2762515}

\bibitem[Liu02]{Liu_AlgGeomAndArithCurves}
Qing Liu, \emph{Algebraic geometry and arithmetic curves}, Oxford University
  Press, 2002.

\bibitem[Mat80]{matsumura}
Hideyuki Matsumura, \emph{Commutative algebra}, 2 ed., Benjamin/Cummings, 1980.

\bibitem[Nag75]{nagata-local}
Masayoshi Nagata, \emph{Local rings}, Robert E. Krieger Publishing Co.,
  Huntington, N.Y., 1975, Corrected reprint. \MR{0460307}

\bibitem[Poo17]{poonen-qpoints}
Bjorn Poonen, \emph{Rational points on varieties}, Graduate Studies in
  Mathematics, vol. 186, American Mathematical Society, Providence, RI, 2017.
  \MR{3729254}

\bibitem[Pop10]{Pop-henselian}
Florian Pop, \emph{Henselian implies large}, Annals of Mathematics \textbf{172}
  (2010), no.~3, 2183--2195.

\bibitem[Pop14]{Pop-little}
\bysame, \emph{Little survey on large fields - old {\&} new}, Valuation Theory
  in Interaction, European Mathematical Society Publishing House, 2014,
  pp.~432--463.

\bibitem[PZ78]{Prestel1978}
Alexander Prestel and Martin Ziegler, \emph{Model theoretic methods in the
  theory of topological fields.}, Journal für die reine und angewandte
  Mathematik \textbf{0299\_0300} (1978), 318--341.

\bibitem[Rot97]{Rotthaus-excellent}
Christel Rotthaus, \emph{Excellent rings, henselian rings, and the
  approximation property}, Rocky Mountain Journal of Mathematics \textbf{27}
  (1997), no.~1.

\bibitem[Sri19]{Srinivasan}
Padmavathi Srinivasan, \emph{A virtually ample field that is not ample}, Israel
  Journal of Mathematics \textbf{234} (2019), no.~2, 769--776.

\bibitem[{Sta}20]{stacks-project}
The {Stacks Project Authors}, \emph{\textit{Stacks Project}},
  \url{https://stacks.math.columbia.edu}, 2020.

\bibitem[Tem13]{Temkin_InsepLocalUnif}
Michael Temkin, \emph{Inseparable local uniformization}, J. Algebra
  \textbf{373} (2013), 65--119. \MR{2995017}

\bibitem[WY23]{secondpaper}
Erik Walsberg and Jinhe Ye, \emph{\'{E}z fields}, J. Algebra \textbf{614}
  (2023), 611--649. \MR{4499357}

\end{thebibliography}

\end{document}